\documentclass[11pt, reqno]{amsart}
\usepackage{color}
\usepackage{amsmath}
\usepackage{latexsym,amssymb}
\usepackage[mathscr]{euscript}
\usepackage{stmaryrd}
\usepackage{bm}
\usepackage{enumerate}
\usepackage{setspace}
\newtheorem{thm}{Theorem}[section]
\newtheorem{proposition}{Proposition}[section]
\newtheorem{example}{\bf Example}[section]

\newtheorem{lemma}{\bf Lemma}[section]
\newtheorem{definition}{Definition}[section]
\numberwithin{equation}{section}
\newtheorem{assumption}{Assumption}[section]

\begin{document}

\baselineskip=17pt

\title[]
{ risk-sensitive discounted  cost criterion for Continuous-time Markov decision processes on a general state space}

\author[Subrata Golui]{Subrata Golui}
\address{Department of Mathematics\\
Indian Institute of Technology Guwahati\\
Guwahati, Assam, India}
\email{golui@iitg.ac.in}

\author[Chandan Pal]{Chandan Pal}
\address{Department of Mathematics\\
Indian Institute of Technology Guwahati\\
Guwahati, Assam, India}
\email{cpal@iitg.ac.in}



\date{}

\begin{abstract}
\vspace{2mm}
\noindent
In this paper, we consider  risk-sensitive discounted control problem for continuous-time jump Markov processes taking values in general state space. The transition rates of underlying continuous-time jump Markov processes and the cost rates are allowed to be unbounded. Under certain Lyapunov condition, we establish the existence and uniqueness of the solution to the Hamilton-Jacobi-Bellman (HJB) equation. Also we prove the existence of optimal risk-sensitive control in the class of Markov control.

\vspace{2mm}


\end{abstract}

\maketitle

\section{INTRODUCTION}
In this paper we study the risk-sensitive discounted criterion for  continuous-time Markov decision processes (CTMDPs) with Borel state space. Risk-sensitive or exponential of integral’ is a very popular cost criterion due to its applications in many areas such as queueing systems and finance, for more details see [\cite{BR}, \cite{WH}] and the references therein. In the literature risk-sensitive control problems for CTMDPs are an important class of stochastic optimal control problems and have been widely studied under different sets of conditions. Finite horizon risk-sensitive CTMDPs for countable sate space were studied in [\cite{GS}, \cite{GHH}, \cite{GLZ}, \cite{W}] and for infinite horizon risk-sensitive CTMDPs we refer to [\cite{GS}, \cite{GZ2}, \cite{KP1}, \cite{KP2}, \cite{PP}, \cite{Z1}]. For important contributions to the risk-sensitive control of discrete time MDP on a general state space, see  [\cite{MS1}, \cite{MS2}]. Although risk-sensitive control of CTMDPs on a countable state space have been studied extensively, but the corresponding literature in the context of  risk sensitive control of CTMDPs on general state space is rather limited. Some exceptions are \cite{GZ},\cite{PP}.

  In the paper \cite{PP}, the authors studied risk-sensitive control of pure jump processes on general state space. They considered bounded transition and cost rates and all controls are Markovian. In \cite{PP}, authors proved a HJB characterization of the optimal risk-sensitive control. The boundedness assumption on transition and cost rates plays a key role in the proof of the existence of the optimal risk-sensitive control in \cite{PP}. This boundedness requirement, however, imposes some restrictions in applications, for instance in queueing control and population processes, where the transition and reward/cost rates are usually unbounded. In \cite{GZ}, the author considered the finite-horizon risk-sensitive control problem for CTMDPs on Borel state space with unbounded transition and cost rates and proved the existence of optimal control via HJB equation.
  
   In this paper we study a much more general problem.  To the best of our knowledge, this is the first work which deals with infinite horizon discounted risk-sensitive control for CTMDPs on general state space with unbounded cost and transition rates and the controls can be  history-dependent. The main objective of this work is to prove the existence of solution of the HJB equation and characterization of optimal risk-sensitive control.
We first consider for bounded transition and cost rates, and establish the existence of a solution to the corresponding  HJB equation by Banach's fixed point theorem as in \cite{PP}. Then we will relax the bounded hypothesis and we extend this result to unbounded transition and cost rates. We characterize the value function via HJB equation. Also we prove the existence of an optimal control in the class of Markov control and HJB characterization of the optimal risk-sensitive control.
 
The rest of this article is structured as follows. Section 2 deals with the description of the problem, required notations, some Assumptions, and preliminary results. In Section 3, we gave continuity-compactness Assumption and we prove the stochastic representation of the solution of the HJB equation (\ref{eq 3.1}). In Section 4, we truncate our transition and cost rates and prove the existence of the unique solution to the HJB equation. The required optimal control is proven in Section 5. In Section 6, we illustrate our theory and assumptions by an illustrative example.
\section{The control problems}
The model of CTMDP is a five-tuple which consists of the following elements: 
\begin{equation*}
\mathbb{M}:=\{S, A, (A(x)\subset A, x\in S),c(x, a),q( \cdot|x, a)\}, \label{eq 2.1}
\end{equation*}
\begin{itemize}
\item  a Borel space $S$, called the state space, whose elements are referred to as states of the system and the corresponding Borel $\sigma$-algebra is $\mathcal{B}(S)$. 	
\item   $A$ is the action set, which is assumed to be Borel space with the Borel $\sigma$-algebra $\mathcal{B}(A)$.
\item  for each $x\in S$, $A(x)\in \mathcal{B}(A)$ denotes the set of admissible actions for state $x$. Let $K:=\{(x, a)|x\in S, a\in A(x)\}$, which is a Borel subset of $S\times A$.
\item  the measurable function $c:K \to \mathbb{R}_{+}$ denotes the cost rate function. We require cost function $c(x, a)$ to measure (or evaluate) the utility of taking action $a$ at state $x$.
	\item given any $(x, a)\in K$, the transition rate $q(\cdot | x, a)$ is a Borel measurable signed kernel on $S$ given $K$. That is, $q(\cdot|x,a)$ satifies countable additivity; $q(D| x, a)\geq 0 $ where $(x,a)\in K$ and $x\notin D$. Moreover, we assume that $q(\cdot | x, a)$ satisfies the following conservative and stable conditions: for any $x\in S,$ 
			\begin{align*}
	&q(S|x,a)\equiv 0 ~~~\text{and}\nonumber\\
	&~q^{*}(x):=\sup_{a\in A(x)}q_x(a)<\infty,
	\end{align*} 
	where $q_x(a):=-q(\{x\}| x, a)\geq 0.$ We need transition rates to specify the random dynamic evolution of the system.
\end{itemize}
Next, we give an informal description of the evolution of the CTMDPs as follows. The controller observes continuously the current state of the system. When the system is in state $x\in S$ at time $t\geq0$, he/she chooses action $a_t\in A(x)$ according to some control. As a consequence of this, the following happens: 
\begin{itemize}
\item the controller incurs an immediate  cost at rate $c(x, a_t)$; and
\item after a random sojourn time (i.e., the holding time at state $x$), the system jumps to a set $B$ ($x\notin B$) of states with the transition probability $\dfrac{q(B|x,a_t)}{q_x(a_t)}$ determined by the transition rates $q(dy|x,a_t)$. The distribution function of the sojourn time is $(1-e^{-\int_{t}^{t+x}q_x(a_s)ds})$. (see Proposition B.8 in [\cite{GH2}, p. 205] for details).
\end{itemize} 
When the state of the system transits to the new state $y\neq x$, the above procedure is repeated.
Thus, the controller tries to minimize his/her costs with respect to some performance criterion $\mathscr{J}_\alpha(\cdot,\cdot, \cdot)$, which in our present case is defined by (\ref{eq 2.5}), below. 
To formalize what is described above, below we describe the construction of continuous time Markov decision processes (CTMDPs) under possibly history-dependent controls.
 To construct the underlying CTMDPs (as in [\cite{GP}, \cite{K}, \cite{PZ}], \cite{PZ1}) we introduce some notations: let $S_\Delta:=S \cup \{\Delta\}$ (with some $\Delta \notin S$), $\Omega_0:=(S\times(0,\infty))^\infty$, $\Omega_k:=(S\times (0,\infty))^k\times S\times (\{\infty\}\times\{\Delta\})^\infty$ for $k\geq 1$ and $\Omega:=\cup_{k=0}^\infty\Omega_k$. Let $\mathscr{F}$ be the Borel $\sigma$-algebra on $\Omega$. Then we obtain the measurable space $(\Omega, \mathscr{F})$. 
	For some $k\geq 1$, and sample $ \omega:=(x_0, \theta_1, x_1, \cdots , \theta_k, x_k, \cdots)\in \Omega,$ define
	 \begin{align*}
	X_0(\omega):=x_0,~ T_0(\omega):=0,~ X_k(w):=x_k,~ T_n(\omega):= T_{n-1}(\omega)+\theta_{n},~ T_\infty(\omega):=\lim_{n\rightarrow\infty}T_n(\omega).
	\end{align*}
 Using $\{T_k\}$, we define the state process $\{\xi_t\}_{t\geq 0}$ as
	\begin{equation}
\xi_t(\omega):=\sum_{k\geq 0}I_{\{T_k\leq t<T_{k+1}\}}x_k+ I_{\{t\geq T_\infty\}}\Delta, \text{ for } t\geq 0~(\text{with}~ T_0:=0).\label{eq 2.4}
	\end{equation}
	Here, $I_{E}$ denotes the indicator function of a set $E$, and we use the convention that $0+z=:z$ and $0z=:0$ for all $z\in S_\Delta$. Obviously, $\xi_t(\omega)$ is right-continuous on $[0,\infty)$. We denote $\xi_{t-}(\omega):=\liminf_{s\rightarrow t-}\xi_s(\omega)$. From eq. (\ref{eq 2.4}), we see that $T_k(\omega)$ $(k\geq 1)$ denotes the $k$-th jump moment of $\{\xi_t, t\geq 0\}$, $X_{k-1}(\omega)=x_{k-1}$ is the state of the process on $[T_{k-1}(\omega),T_k(\omega))$, $\theta_k=T_k(\omega)-T_{k-1}(\omega)$ plays the role of sojourn time at state $x_{k-1}$, and the sample path $\{\xi_t(\omega),t\geq 0\}$ has at most denumerable states $x_k(k=0,1,\cdots)$.  The process after $T_\infty$ is regarded to be absorbed in the state $\Delta$. Thus, let $q(\cdot | \Delta, a_\Delta):\equiv 0$, $A_\Delta:=A\cup \{a_\Delta\}$, $ A(\Delta):=\{a_\Delta\}$, $c(\Delta, a):\equiv 0$ for all $a\in A_\Delta$, where $a_\Delta$ is isolated point.\\
	To precisely define the criterion, we need to introduce the concept of a control as in [\cite{GP},\cite{GHS} and \cite{KR}]. Take the right-continuous $\sigma$-algebras $\{\mathscr{F}_t\}_{t\geq 0}$ with $\mathscr{F}_t:=\sigma(\{T_k\leq s,X_k\in B\}:B\in \mathcal{B}(S), 0\leq s\leq t, k\geq0)$. For all $t\geq 0$, $\mathscr{F}_{s-}=:\bigvee_{0\leq t<s}\mathscr{F}_t$, and $\mathscr{P}:=\sigma(\{A\times \{0\},A\in \mathscr{F}_0\} \cup \{ B\times (s,\infty),B\in \mathscr{F}_{s-}\})$ which denotes the $\sigma$-algebra of predictable sets on $\Omega\times [0,\infty)$ related to $\{\mathscr{F}_t\}_{t\geq 0}$.
To complete the specification of a stochastic optimal control problem, we need, of course, to introduce an optimality criterion. This requires to define the
class of controls as below.
\begin{definition}
	A transition probability $\pi(da|\omega,t)$ from $(\Omega\times[0,\infty),\mathscr{P})$ onto $(A_\Delta,\mathcal{B}(A_\Delta))$ such that $\pi(A(\xi_{t-}(\omega)|\omega,t)\equiv 1$ is called a history-dependent control. The set of all randomized history-dependent controls is denoted by $\Pi$.
	A control $\pi\in \Pi$, is called a Markov if $\pi(da | \omega,t)=\pi(da | \xi_{t-}(w),t)$ for every $w\in \Omega$ and $t\geq 0$, where $\xi_{t-}(w):=\lim_{s\uparrow t}\xi_s(w)$. We denote by  $\Pi^{m}$ the family of all Markov controls. A Markov control $\pi_t(da|\cdot)$ is called a deterministic Markov control whenever there exists a measurable mapping $f:[0,\infty)\times S\rightarrow A$ such that $\pi(da|t,x)=I_{\{f(t,x)\}}(da)$, which means that $\pi(da|t,x)$ is a Dirac measure at $f(t,x)$ for every $x\in S$ and $t\geq 0$. Such a Markov control will be denoted by $f$ for simplicity. The set of such controls is denoted by $\Pi^d_m$. 
	\end{definition}
For any compact metric space $Y$, let $P(Y)$ denote the space of probability measures on $Y$ with Prohorov topology.
Under Assumption \ref{assm 2.1} below, for any initial state $x\in S$ and any control $\pi\in \Pi$, Theorem 4.27 in \cite{KR} yields the existence of a unique probability measure denoted by $P^{\pi}_x$ on $(\Omega,\mathscr{F})$.  Let $E^{\pi}_x$ be the expectation operator with respect to  $P^{\pi}_x$.
Fix any discounted factor $\alpha>0$. For any $\pi\in\Pi$ and $x\in S$, the risk-sensitive discounted criterion is defined as
\interdisplaylinepenalty=0
\begin{align}
\mathscr{J}_\alpha(\theta,x,\pi):=\frac{1}{\theta}\log\biggl\{{E}^{\pi}_x\biggl[exp\biggl(\theta{\int_{0}^{\infty}e^{-\alpha t}\biggl(\int_{A}c(\xi_t(\omega),a)\pi(da|\omega,t)\biggr)dt}\biggr)\biggr]\biggr\},\label{eq 2.5}
\end{align}
provided that the integral  is well defined, where $\xi_t$ is the Markov process corresponding to $\pi\in\Pi$
and $\theta\in (0,1]$ denotes a risk-sensitive parameter and the limiting case of $\theta\rightarrow 0$ is the risk-neutral case.
For each $x\in S$, let
$$\mathscr{J}^{*}_\alpha(\theta,x)=\inf_{\pi\in \Pi}\mathscr{J}_\alpha(\theta,x,\pi).$$
A control $\pi^{*}\in\Pi$ is said to be optimal if $\mathscr{J}_\alpha(\theta,x,\pi^{*})=\mathscr{J}^{*}_\alpha(\theta,x)$ for all $x\in S$.
The objective of this paper is to provide conditions for the existence of optimal control and introduce a HJB characterization of such control.\\
Since logarithm is an increasing function, instead of studying $\mathscr{J}_\alpha(\theta,x,\pi)$, we will consider $\tilde{J}_\alpha(\theta,x,\pi)$ on $[0,1]\times S\times \Pi$ defined by
\begin{align}
\tilde{J}_\alpha(\theta,x,\pi):= {E}^{\pi}_x\biggl[exp\biggl(\theta{\int_{0}^{\infty}e^{-\alpha t}\biggl(\int_{A}c(\xi_t(\omega),a)\pi(da|\omega,t)\biggr)dt}\biggr)\biggr].\label{eq 2.6}
\end{align}
Obviously, $\tilde{J}_\alpha(\theta,x,\pi)\geq 1$ for $(\theta,x)\in[0,1]\times S$ and $\pi\in\Pi$, and we have
$\pi^{*}$ is optimal if and only if
$\inf_{\pi\in \Pi}\tilde{J}_\alpha(\theta,x,\pi)=\tilde{J}(\theta,x,\pi^{*})
=:\tilde{J}^{*}_\alpha(\theta,x) ~\forall x\in S.$
Since the rates $q(dy|x,a)$ and costs $c(x,a)$ are allowed to be unbounded, we next give conditions for the non-explosion of $\{\xi_t,t\geq 0\}$ and finiteness of $ \mathscr{J}_\alpha (\theta , x, \pi )$, which had been widely used in CTMDPs; see, for instance,
[ \cite{GH2}, \cite{GHS}, \cite{GL}, \cite{GP} and \cite{PRH}] and reference therein.
\begin{assumption}\label{assm 2.1}
	There exists a real-valued Borel measurable function $V_0 \geq 1$ on $S$ and constants  $\rho_0> 0$, $M_0>0$, $L_0\geq 0$ and $0<\rho_1<\min\{\alpha,\rho^{-1}_0\alpha^2\}$ such that
	\begin{enumerate}
		\item [(i)] $\int_{S}V_0(y)q(dy |x, a)\leq \rho_0 V_0(x)~~~\forall (x, a)\in K$;
		
		\item [(ii)] $\sup_{a\in A(x)}q_x(a)\leq M_0 V_0(x)~~~\forall x\in S$;
		\item [(iii)] 	
		$\sup_{a\in A(x)}c(x,a)\leq \rho_1\log V_0(x)+L_0~~~\forall x\in S.$
	\end{enumerate}
\end{assumption}
\begin{proposition}\label{prop 2.1}
Under Assumption \ref{assm 2.1}, for any control $\pi \in \Pi$ and $(\theta,x)\in [0,1]\times S$, the following results are true:
	\begin{enumerate}
		\item [(a)] ${P}^{\pi}_x(T_\infty=\infty)=1$, $P^\pi_x(\xi_0=x)=1$, and ${P}^{\pi}_x(\xi_t\in S)=1$ for all $t\geq 0$;
				\item [(b)] ${E}^{\pi}_x[V_0(\xi_t)]\leq e^{\rho_0 t} V_0(x)$~\text{for all}~$t\geq 0;$
				\item [(c)]  We have
				\begin{align*}
				\tilde{J}_\alpha(\theta,x,\pi)\leq \frac{\alpha^2}{\alpha^2-\rho_0\rho_1\theta}e^{ {\theta L_0}/{\alpha}}[V_0(x)]^{\frac{\rho_1\theta}{\alpha}}\leq \frac{\alpha^2}{\alpha^2-\rho_0\rho_1}e^{{L_0}/{\alpha}}V_0(x).
				\end{align*}
				Also, we get
				\begin{align}
				\mathscr{J}^{*}_\alpha(\theta,x)\leq \log\biggl({\frac{\alpha^2}{\alpha^2-\rho_0\rho_1}}\biggr)+\frac{L_0}{\alpha}+\frac{\rho_1}{\alpha}\log{V_0(x)} ~~\forall \theta\in (0,1],x\in S.\label{eq 2.7}
				\end{align}
			\end{enumerate}
		\end{proposition} 
		\begin{proof}
 For parts $(a)$ and $(b)$, see, \cite{GHS} and ( \cite{GP}, Theorem 3.1 ).\\
    Proof of part (c): 	Let $\beta(dt)=\alpha e^{-\alpha t}dt$, which is a probability measure on $[0,\infty).$ For any $\pi\in\Pi$ and $(\theta,x)\in[0,1]\times S$, by (\ref{eq 2.6}) and Jensen's inequality we have
    \interdisplaylinepenalty=0
    \begin{align*}
    \tilde{J}_\alpha(\theta,x,\pi)&= {E}^{\pi}_x\biggl[exp\biggl(\int_{0}^{\infty}\frac{\theta}{\alpha}\int_{A}c(\xi_t(\omega),a)\pi(da|\omega,t)\beta(dt)\biggr)\biggr]\\
    &\leq {E}^{\pi}_x\biggl[\int_{0}^{\infty}exp\biggl(\frac{\theta}{\alpha}\int_{A}c(\xi_t(\omega),a)\pi(da|\omega,t)\biggr)\beta(dt)\biggr].
    \end{align*}
    By Assumption \ref{assm 2.1} and part (b) we obtain
    \interdisplaylinepenalty=0
    \begin{align*}
    \tilde{J}_\alpha(\theta,x,\pi) &\leq {E}^{\pi}_x\biggl[\int_{0}^{\infty}exp\biggl(\frac{\theta}{\alpha}(\rho_1\log{V_0(\xi_t)}+L_0)\biggr)\beta(dt)\biggr]\\
    &=e^{{\theta L_0}/{\alpha}}\biggl[\int_{0}^{\infty}{E}^{\pi}_x\biggl(V_0(\xi_t)^{\frac{\rho_1\theta}{\alpha}}\biggr)\beta(dt)\biggr]\\
    &\leq e^{{\theta L_0}/{\alpha}}\biggl[\int_{0}^{\infty}({E}^{\pi}_x[V_0(\xi_t)])^{\frac{\rho_1\theta}{\alpha}}\beta(dt)\biggr]~~~(\text{since}~~ \rho_1\theta<\alpha)\\
    &\leq \alpha e^{{\theta L_0}/{\alpha}} [V_0(x)]^{\frac{\rho_1\theta}{\alpha}}\biggl[\int_{0}^{\infty}exp\biggl(\frac{\rho_0\rho_1\theta t}{\alpha}-\alpha t\biggr)dt\biggr]\\
    &={\frac{\alpha^2}{\alpha^2-\rho_0\rho_1\theta}} e^{{\theta L_0}/{\alpha}}[V_0(x)]^{\frac{\rho_1\theta}{\alpha}}~~~(\text{since}~~\rho_0\rho_1\theta<\alpha^2).
    \end{align*}
    Hence, we have $\sup_{\theta\in [0,1]}\tilde{J}^{*}_\alpha(\theta,x)\leq \frac{\alpha^2}{\alpha^2-\rho_0\rho_1} e^{{ L_0}/{\alpha}}V_0(x)$, and 
    \interdisplaylinepenalty=0
    \begin{align*}
    \sup_{\theta\in (0,1]}\mathscr{J}^{*}_\alpha(\theta,x)= \sup_{\theta\in (0,1]}\frac{1}{\theta}\log{\tilde{J}^*_\alpha(\theta,x)}\leq \sup_{\theta\in (0,1]}\frac{1}{\theta}\biggl(\log{\frac{\alpha^2}{\alpha^2-\rho_0\rho_1\theta}}\biggr)+\frac{L_0}{\alpha}+\frac{\rho_1}{\alpha}\log V_0(x).
    \end{align*}
   Now, by a direct calculation, one can show that (\ref{eq 2.7}) holds. 
\end{proof}
Here we assume the following conditions, so that we can apply the Feynman-Kac formula formula for  a large enough class of functions, which had been widely used in CTMDPs; see, for instance, [\cite{GHH}, \cite{GL}, \cite{GLZ}, \cite{GZ}, \cite{W}].
\begin{assumption}\label{assm 2.2}
There exist a Borel measurable function $V_1\geq 1$ on $S$, and constants $0<\rho_2<\alpha$, $b_1\geq 0$, $M_1\geq 1$ such that
\begin{enumerate}
	\item [(i)]	
	$\int_{S}q(dy|x,a)V^2_1(y)\leq \rho_2V^2_1(x)+b_1~~\forall~(x,a)\in K$,
	\item [(ii)] $V_0^2(x)\leq M_1 V_1(x)~~\forall~x\in S$ where $V_0$ is introduced in Assumption \ref{assm 2.1}.
\end{enumerate}
\end{assumption}
We now introduce some frequently used notations. 
\begin{itemize}
	\item For any Borel space $X$, $\mathcal{B}(X)$ denotes the corresponding Borel $\sigma$-algebra.
	\item $C^\infty_c(a,b)$ denotes the set of all infinitely differentiable functions on $(a,b)$ with compact support.
	\item Let $A_c([0,1]\times S)$ denote the space of all functions which are real-valued and differentiable almost everywhere with respect to $\theta\in[0,1]$eg. When the partial derivative (with respect to $\theta\in [0,1]$) does not exist for some $(\theta,x)\in [0,1]\times S$, we take $\frac{\partial\varphi_\alpha}{\partial \theta}(\theta,x)$ to be any real number, and so $\frac{\partial\varphi_\alpha}{\partial \theta}(\cdot,\cdot)$ is defined on $[0,1]\times S$. Given any real-valued function $W\geq 1$ on $S$ and any Borel set $X$, a real-valued function $u$ on $X\times S$ is called $W$ bounded if $\|u\|^\infty_{W}:=\sup_{(\theta,x)\in X\times S}\frac{|u(\theta,x)|}{W(x)}< \infty$. Denote $B_{W}(X\times S)$ the Banach space of all $W$-bounded functions. When $W\equiv1$, $B_{1}([0,1]\times S)$ is the space of all bounded functions on $[0,1]\times S.$\\
	Now define $B^1_{W_0,W_1}([0,1]\times S):=\{\varphi_{\alpha}\in B_{W_0}([0,1]\times S)\cap A_c([0,1]\times S):\frac{\partial\varphi_\alpha}{\partial \theta}\in B_{W_1}([0,1]\times S)\}$. 
\end{itemize}

	
\section{stochastic representation of a solution to the HJB equation}
In this section, we prove that if the HJB equation for the cost criterion (\ref{eq 2.6}) has a solution then we will give a stochastic representation of that solution. Using dynamic programming heuristics, the HJB equations for the discounted cost criterion (\ref{eq 2.6}) is given by
\interdisplaylinepenalty=0
\begin{align}
\left\{ \begin{array}{ll}\alpha\theta \frac{\partial\varphi_\alpha}{\partial \theta}(\theta,x)&=\displaystyle{\inf_{a\in A(x)}\biggl[\int_{S}q(dy|x,a)\varphi_\alpha(\theta,y)+\theta c(x,a)\varphi_\alpha(\theta,x)\biggr]},\\&1\leq \varphi_\alpha(\theta,x)\leq {\frac{\alpha^2}{\alpha^2-\rho_0\rho_1\theta}} e^{{\theta L_0}/{\alpha}}(V_0(x))^{\frac{\rho_1\theta}{\alpha}}~~\text{for}~~(\theta,x)\in[0,1]\times S,\label{eq 3.1}
\end{array}\right.
\end{align}
 for each $x\in S$ and a.e. $\theta\in [0,1]$ where the upper bound of $\varphi_\alpha(\theta,x)$  is inspired by Proposition \ref{prop 2.1}.
	
To ensure the existence of optimal control, in addition to Assumptions \ref{assm 2.1} and  \ref{assm 2.2}, we also need the following continuity and compactness conditions.
\begin{assumption}\label{assm 3.1}
	The following conditions hold:
	\begin{enumerate}
		\item [(i)] for each $x\in S$, the set $A(x)$ is compact;
		\item  [(ii)]for any fixed $x\in S$, $q(\cdot | x, a)$ and $c(x, a)$ are continuous in $a\in A(x)$;
		
		\item [(iii)]for any given $x\in S$, the function $\displaystyle \int_{S}V_0(y)q(dy|x,a)$ is continuous in $a\in A(x)$, where $V_0$ is introduced in Assumption \ref{assm 2.1}.
	\end{enumerate}
\end{assumption}

 In the next theorem we show that if the HJB equation has a solution then its stochastic representation is equal to the value function corresponding to the cost criterion (\ref{eq 2.6}).
\begin{thm}\label{theo 3.1}
Under Assumptions \ref{assm 2.1}, \ref{assm 2.2}, and \ref{assm 3.1} suppose that the HJB equation (\ref{eq 3.1}) has a solution $\varphi_\alpha\in B^1_{V_0,V_1}([0,1]\times S)$ satisfying the bounds. Then, for all $(\theta,x)\in [0,1]\times S$, we have the probabilistic representation of $\varphi_\alpha$ as
	\interdisplaylinepenalty=0
	\begin{align}
	\varphi_\alpha(\theta,x)=\inf_{\pi\in \Pi}E^{\pi}_x\biggl[exp\biggl(\theta\int_{0}^{\infty}\int_{A}e^{-\alpha t}c(\xi_t,a)\pi(da|\omega,t)dt\biggr)\biggr]\label{eq 3.2}
	\end{align}
	which means $\varphi_\alpha(\theta,x)=\tilde{J}^{*}_\alpha(\theta,x)$ for all $(\theta,x)\in [0,1]\times S$.
	\end{thm}
\begin{proof} 
	First we see that
	\begin{align*}
\biggl[\theta c(x,a)\varphi_\alpha(\theta,x)+\int_{S}q(dy|x,a)\varphi_\alpha(\theta,y)\biggr] 
	\end{align*}
	is continuous in $a\in A(x)$ and $A(x)$ is compact. So by measurable selection theorem, [\cite{BS},Proposition 7.33], there exists a measurable function $f^{*}:[0,1]\times S\rightarrow A$ such that 
		\begin{align}
	&\inf_{a\in A(x)}\biggl[\theta c(x,a)\varphi_\alpha(\theta,x)+\int_{S}q(dy|x,a)\varphi_\alpha(\theta,y)\biggr]\nonumber\\
	&=\biggl[\theta c(x,f^{*}(\theta,x))\varphi_\alpha(\theta,x)+\int_{S}q(dy|x,f^{*}(\theta,x))\varphi_\alpha(\theta,y)\biggr]. \label{eq 3.3} 
	\end{align}
	Let
	\begin{equation*}
	\pi^{*}:\mathbb{R}_+\times S \to P(A)  \
	\end{equation*}
	be defined by
	\begin{align*}
	&\pi^{*}(\cdot|t,x)=I_{\{\hat{f}^{*}(t,x)\}}(\cdot),~\text{where}~\hat{f}^{*}:[0,1]\times S\rightarrow A,\\&\quad~\text{be a measurable mapping, defined by}~ \hat{f}^{*}(t,x):=f^{*}(\theta e^{-\alpha t},x).
	\end{align*}
	When $\frac{\partial\varphi_{\alpha}}{\partial\theta}$ does not exist for some $(\theta,x)$, we define
	\begin{align*}
	\alpha\theta\frac{\partial\varphi_\alpha}{\partial\theta}(\theta,x)
	&= \inf_{a\in A(x)}\displaystyle\biggl[ \theta c(x,a)\varphi_\alpha(\theta,x)+\int_{S}q(dy|x,a)\varphi_\alpha(\theta,y)\biggr].
	\end{align*}
Then we observe from equation (\ref{eq 3.1}) that for any $(\theta,x)\in [0,1]\times S$ and $a\in A(x)$ that
	\begin{align}
	-\alpha\theta \frac{\partial\varphi_\alpha}{\partial \theta}(\theta,x)+\displaystyle{\biggl[\int_{S}q(dy|x,a)\varphi_\alpha(\theta,y)+\theta c(x,a)\varphi_\alpha(\theta,x)\biggr]}\geq 0\label{eq 3.4}.
	\end{align}
	For any history-dependent control $\pi\in\Pi$ and $\theta\in [0,1]$,
let $\{\xi_t, t\geq 0\}$ be the corresponding process, and define  $\theta(t):=\theta e^{-\alpha t}$. Now for each $\omega\in\Omega$, by equation (\ref{eq 3.4}), we get
\begin{align}
-\alpha\theta(s)& \frac{\partial\varphi_\alpha}{\partial \theta}(\theta(s),\xi_s(\omega))\nonumber\\&+\displaystyle{\biggl[\int_{S}\int_{A}q(dy|\xi_s(\omega),a)\varphi_\alpha(\theta(s),y)\pi(da|\omega,s)+\theta(s) \int_{A}c(\xi_s(\omega),a)\varphi_\alpha(\theta(s),\xi_s(\omega))\pi(da|\omega,s)\biggr]}\geq 0\label{eq 3.5}
\end{align}
  and define $g: [0,\infty)\times S\times\Omega \to [0, \infty)$ by
$$g(t,x,\omega):=exp\biggl(\int_{0}^{t}\int_{A}\theta(s)c(\xi_s(\omega),a)\pi(da|\omega,s)ds\biggr)\varphi_\alpha(\theta(t),x).$$  
Let $\beta(ds):=\alpha e^{-\alpha s}ds$. Then under Assumptions \ref{assm 2.1} and \ref{assm 2.2}, we have
\interdisplaylinepenalty=0
\begin{align}
&E^{\pi}_x\biggl[exp\biggl(\int_{0}^{t}\int_{A} 2e^{-\alpha s}c(\xi_s,a)\pi(da|\omega,s)ds\biggr)\biggr]\nonumber\\
&\leq E^{\pi}_x\biggl[exp\biggl(\int_{0}^{\infty}\frac{2}{\alpha}\int_{A} c(\xi_s,a)\pi(da|\omega,s)\beta(ds)\biggr)\biggr]\nonumber\\
&\leq E^{\pi}_x\biggl[\int_{0}^{\infty}exp\biggl( \frac{2}{\alpha}\int_{A}c(\xi_s,a)\pi(da|\omega,s)\biggr)\beta(ds)\biggr]\nonumber\\
&~(\text{by Jensen's inequality})\nonumber\\
&\leq E^{\pi}_x\biggl[\int_{0}^{\infty}exp\biggl(\frac{2}{\alpha}(\rho_1\log{V_0(\xi_s)}+L_0)\biggr)\beta(ds)\biggr]\nonumber\\
&~(\text{by Assumption  \ref{assm 2.1}})\nonumber\\
&\leq e^{{2L_0}/{\alpha}}\biggl[\int_{0}^{\infty}E^{\pi}_x\biggl(V_0(\xi_s)^{\frac{4\rho_1}{\alpha}}\biggr)\beta(ds)\biggr]\nonumber\\
&\leq e^{{2L_0}/{\alpha}}M_1^2\biggl[\int_{0}^{\infty}E^{\pi}_x\biggl(V_1(\xi_s)^{\frac{2\rho_1}{\alpha}}\biggr)\beta(ds)\biggr]\nonumber\\
&\leq \alpha e^{{2L_0}/{\alpha}}M_1^2\biggl(V_1^2(x)+\frac{b_1}{\rho_2}\biggr)\biggl[\int_{0}^{\infty}e^{\rho_2 s-\alpha s}ds\biggr]\nonumber\\
&\leq\frac{\alpha e^{{2L_0}/{\alpha}}}{\alpha-\rho_2}M_1^2\biggl(V_1^2(x)+\frac{b_1}{\rho_2}\biggr).\label{eq 3.6}
\end{align}
Now using Assumptions \ref{assm 2.1} and \ref{assm 2.2}, we obtain
\begin{align}
&\biggl|\biggl[-\alpha \theta(s)\frac{\partial\varphi_\alpha}{\partial\theta}(\theta(s),\xi_s(\omega))+\int_{S}\int_{A}q(dy|\xi_s(\omega),a)\varphi_\alpha(\theta(s),y)\pi(da|\omega,s)\nonumber\\
&\quad+\theta(s)\int_{A}c(\xi_s(\omega),a)\varphi_\alpha(\theta(s),\xi_s(\omega))\pi(da|\omega,s)\biggr]\biggr|\nonumber\\
&\leq \biggl[\alpha \biggl\|\frac{\partial\varphi_\alpha}{\partial\theta}\biggr\|_{V_1}V_1(\xi_s)+\|\varphi_{\alpha}\|_{V_0}\biggl(\int_{S}\int_{A}q(dy|\xi_s(\omega),a)V_0(y)\pi(da|\omega,s)+2V_0^2(\xi_s)M_0\biggr)\nonumber\\
&\quad+\|\varphi_{\alpha}\|_{V_0}(\rho_1\log V_0(\xi_s)+L_0)V_0(\xi_s)\biggr]\nonumber\\
&\leq  \biggl[\alpha \biggl\|\frac{\partial\varphi_\alpha}{\partial\theta}\biggr\|_{V_1}V_1(\xi_s)+\|\varphi_{\alpha}\|_{V_0}(\rho_0V_0(\xi_s)+2V_0^2(\xi_s)M_0)+(\rho_1V_0(\xi_s)+L_0)\|\varphi_{\alpha}\|_{V_0}V_0(\xi_s)\biggr]\nonumber\\
&\leq  \biggl[\alpha \biggl\|\frac{\partial\varphi_\alpha}{\partial\theta}\biggr\|_{V_1}V_1(\xi_s)+\|\varphi_{\alpha}\|_{V_0}(\rho_0+2M_0)M_1V_1(\xi_s)+(\rho_1+L_0)M_1V_1(\xi_s)\|\varphi_{\alpha}\|_{V_0}\biggr]\nonumber\\
&= \biggl[\alpha \biggl\|\frac{\partial\varphi_\alpha}{\partial\theta}\biggr\|_{V_1}+\|\varphi_{\alpha}\|_{V_0}(\rho_0+2M_0+\rho_1+L_0)M_1\biggr]V_1(\xi_s).\label{eq 3.7}
\end{align}

Now by (\ref{eq 3.6}), we get
\begin{align*}
&\biggl|E^{\pi}_x\biggl[exp\biggl(\int_{0}^{s}\int_{A}\theta(v)c(\xi_v(\omega),a)\pi(da|\omega,v)dv\biggr)V_1(\xi_s)\biggr]
\biggr|\nonumber\\
&\leq \biggl|E^{\pi}_x\biggl[exp\biggl(\int_{0}^{s}\int_{A}2\theta(v)c(\xi_v(\omega),a)\pi(da|\omega,v)dv\biggr)\biggr]E^{\pi}_x[V^2_1(\xi_s)]
\biggr|\nonumber\\
&\leq \frac{\alpha e^{{2L_0}/{\alpha}}}{\alpha-\rho_2}M_1^2\biggl(V_1^2(x)+\frac{b_1}{\rho_2}\biggr)e^{\rho_2s}\biggl(V_1^2(x)+\frac{b_1}{\rho_2}\biggr).
\end{align*}
Therefore we have
\begin{align}
&\biggl|E^{\pi}_x\biggl[\int_{0}^{t}exp\biggl(\int_{0}^{s}\int_{A}\theta(v)c(\xi_v(\omega),a)\pi(da|\omega,v)dv\biggr)V_1(\xi_s)ds\biggr]
\biggr|\nonumber\\
&\leq \frac{\alpha e^{{2L_0}/{\alpha}}}{\alpha-\rho_2}M_1^2\biggl(V_1^2(x)+\frac{b_1}{\rho_2}\biggr)\biggl(V_1^2(x)+\frac{b_1}{\rho_2}\biggr)\int_{0}^{t}e^{\rho_2s}ds\nonumber\\
&\leq \frac{\alpha e^{{2L_0}/{\alpha}}}{\alpha-\rho_2}M_1^2\biggl(V_1^2(x)+\frac{b_1}{\rho_2}\biggr)^2\frac{e^{\rho_2t}}{\rho_2}<\infty.\label{eq 3.8}
\end{align}
So by (\ref{eq 3.7}) and (\ref{eq 3.8}), we say
\begin{align*}
&\biggl|E^{\pi}_x\biggl\{\int_{0}^{t}exp\biggl(\int_{0}^{s}\int_{A}\theta(v)c(\xi_v(\omega),a)\pi(da|\omega,v)dv\biggr)\nonumber\\
&\times\biggl[-\alpha \theta(s)\frac{\partial\varphi_\alpha}{\partial\theta}(\theta(s),\xi_s(\omega))+\int_{S}\int_{A}q(dy|\xi_s(\omega),a)\varphi_\alpha(\theta(s),y)\pi(da|\omega,s)\nonumber\\
&+\theta(s)\int_{A}c(\xi_s(\omega),a)\varphi_\alpha(\theta(s),\xi_s(\omega))\pi(da|\omega,s)\biggr]ds\biggr\}\biggr|<\infty.
\end{align*}

Thus, using the extension of Feynman-Kac formula in [ \cite{GLZ}, Theorem 3.1] to the function $g$, we have
\interdisplaylinepenalty=0
\begin{align}
&E^{\pi}_x[g(t,\xi_t(\omega),\omega)]-\varphi_\alpha(\theta,x)\nonumber\\
&=E^{\pi}_x\biggl\{\int_{0}^{t}exp\biggl(\int_{0}^{s}\int_{A}\theta(v)c(\xi_v(\omega),a)\pi(da|\omega,v)dv\biggr)\nonumber\\
&\times\biggl[-\alpha \theta(s)\frac{\partial\varphi_\alpha}{\partial\theta}(\theta(s),\xi_s(\omega))+\int_{S}\int_{A}q(dy|\xi_s(\omega),a)\varphi_\alpha(\theta(s),y)\pi(da|\omega,s)\nonumber\\
&+\theta(s)\int_{A}c(\xi_s(\omega),a)\varphi_\alpha(\theta(s),\xi_s(\omega))\pi(da|\omega,s)\biggr]ds\biggr\}.\label{eq 3.9}
\end{align}
Now from (\ref{eq 3.5}) and (\ref{eq 3.9}), we have 
\begin{align}
\varphi_\alpha(\theta,x)\leq E^{\pi}_x\biggl[exp\biggl(\int_{0}^{t}\int_{A}\theta(s)c(\xi_s,a)\pi(da|\omega,s)ds\biggr)\varphi_\alpha(\theta(t),\xi_t)\biggr].\label{eq 3.10}
\end{align}
Given any $p>1$, let $q>1$ such that $\frac{1}{p}+\frac{1}{q}=1$, by Holder's inequality we have 
\interdisplaylinepenalty=0
\begin{align}
&\varphi_\alpha(\theta,x)\nonumber\\
&\leq E^{\pi}_x\biggl[exp\biggl(\int_{0}^{t}\int_{A}\theta(s)c(\xi_s,a)\pi(da|\omega,s)ds\biggr)\varphi_\alpha(\theta(t),\xi_t)\biggr]\nonumber\\
&\leq \biggl\{E^{\pi}_x\biggl[exp\biggl(p\int_{0}^{t}\int_{A}\theta(s) c(\xi_s,a)\pi(da|\omega,s)ds\biggr)\biggr]\biggr\}^{{1}/{p}}\nonumber\\
&~~~~~~~~~~~~~~~~~~~~~~~\times\biggl\{E^{\pi}_x[\varphi^q_\alpha(\theta(t),\xi_t)]\biggr\}^{{1}/{q}}\nonumber\\
&=:T_1(p,t)\cdot T_2(q,t).\label{eq 3.11}
\end{align}
For $T_2(q,t):=\{E^{\pi}_x[\varphi^q_\alpha(\theta(t),\xi_t)]\}^{{1}/{q}}$, by the upper bound of $\varphi_\alpha$ in (\ref{eq 3.1}), we have 
\begin{align*}
\varphi_\alpha(\theta(t),\xi_t)=\varphi_\alpha(\theta e^{-\alpha t},\xi_t)\leq \frac{\alpha^2}{\alpha^2-\theta e^{-\alpha t}\rho_0\rho_1}exp\biggl(\frac{\theta e^{-\alpha t} L_0}{\alpha}\biggr) [V_0(\xi_t)]^{\frac{\rho_1\theta e^{-\alpha t}}{\alpha}}.
\end{align*}
If $t>\alpha^{-1}\log({\theta q\rho_1}/{\alpha})$ then ${\theta e^{-\alpha t}q\rho_1}/{\alpha}<1$.
Hence, by Jensen's inequality and Proposition \ref{prop 2.1}(b), we obtain
\interdisplaylinepenalty=0
\begin{align}
T_2(q,t)\leq &\biggl\{E^{\pi}_x\biggl[\biggl(\frac{\alpha^2}{\alpha^2-\theta e^{-\alpha t}\rho_0\rho_1}\biggr)^q exp\biggl(\frac{q\theta e^{-\alpha t } L_0}{\alpha}\biggr)  [V_0(\xi_t)]^{\frac{q\rho_1\theta e^{-\alpha t}}{\alpha}}\biggr]\biggr\}^{{1}/{q}}\nonumber\\
&=\frac{\alpha^2}{\alpha^2-\theta e^{-\alpha t}\rho_0\rho_1} exp\biggl(\frac{\theta e^{-\alpha t } L_0}{\alpha}\biggr) \biggl[E^{\pi}_x[V_0^{\frac{q\rho_1\theta e^{-\alpha t}}{\alpha}}(\xi_t)]\biggr]^{\frac{1}{q}}\nonumber\\
&\leq \frac{\alpha^2}{\alpha^2-\theta e^{-\alpha t}\rho_0\rho_1} exp\biggl(\frac{\theta e^{-\alpha t } L_0}{\alpha}\biggr) [E^{\pi}_x(V_0(\xi_t))]^{\frac{\rho_1\theta e^{-\alpha t}}{\alpha}}\nonumber\\
&\leq \frac{\alpha^2}{\alpha^2-\theta e^{-\alpha t}\rho_0\rho_1} exp\biggl(\frac{\theta e^{-\alpha t }}{\alpha}(L_0+\rho_0 \rho_1 t)\biggr) V^{\frac{\theta e^{-\alpha t}\rho_1}{\alpha}}_0(x)=:T_3(t).\label{eq 3.12}
\end{align}
By letting $t \to \infty$ we obtain
\interdisplaylinepenalty=0
\begin{align}
&T_1(p,t)\rightarrow \biggl\{E^{\pi}_x\biggl[exp\biggl(p\int_{0}^{\infty}\int_{A}\theta(s)c(\xi_s,a)\pi(da|\omega,s)ds\biggr)\biggr]\biggr\}^{{1}/{p}}\nonumber\\
&\text{and}~\nonumber\\
&T_3(t)\rightarrow 1.\label{eq 3.13}
\end{align}
 Combining (\ref{eq 3.11}),  (\ref{eq 3.12}) and  (\ref{eq 3.13}), we obtain
 \interdisplaylinepenalty=0
 \begin{align*}
 \varphi_\alpha(\theta,x)\leq \biggl \{E^{\pi}_x\biggl[exp\biggl(p\theta\int_{0}^{\infty}\int_{A}e^{-\alpha t}c(\xi_t,a)\pi(da|\omega,t)dt\biggr)\biggr]\biggr\}^{{1}/{p}},
\end{align*}
 for $p>1$.
 Then, passing to the limit as $p\downarrow 1$, we obtain
 \begin{align*}
 \varphi_\alpha(\theta,x)\leq  E^{\pi}_x\biggl[exp\biggl(\theta\int_{0}^{\infty}\int_{A}e^{-\alpha t}c(\xi_t,a)\pi(da|\omega,t)dt\biggr)\biggr].
 \end{align*}
 Since $\pi\in \Pi$ is arbitrary control, we have
 \interdisplaylinepenalty=0
   \begin{align}
   \varphi_\alpha(\theta,x)&\leq  \inf_{\pi\in \Pi}E^{\pi}_x\biggl[exp\biggl(\theta\int_{0}^{\infty}\int_{A}e^{-\alpha t}c(\xi_t,a)\pi(da|\omega,t)dt\biggr)\biggr].\label{eq 3.14}
   \end{align}
Using (\ref{eq 3.1}), (\ref{eq 3.3}) and (\ref{eq 3.9}), we can show that
 \begin{align}
 &E^{\pi^{*}}_x\biggl[exp\biggl(\int_{0}^{t}\theta(s)c(\xi_s,\pi^{*}_s(da|s,\xi_s))ds\biggr)\varphi_\alpha(\theta(t),\xi_t)\biggr]= \varphi_\alpha(\theta,x).\label{eq 3.15}
 \end{align}
Now, using the lower bound of $\varphi_\alpha$ in (\ref{eq 3.1}) and Fatou's lemma, we obtain
\interdisplaylinepenalty=0
 \begin{align}
&\liminf_{t\rightarrow\infty}E^{\pi^{*}}_x\biggl[exp\biggl(\int_{0}^{t}\int_{A}\theta(s)c(\xi_s,a)\pi^{*}(da|s,\xi_s)ds\biggr)\varphi_\alpha(\theta(t),\xi_t)\biggr]\nonumber\\
&\geq \liminf_{t\rightarrow\infty}E^{\pi^{*}}_x\biggl[exp\biggl(\int_{0}^{t}\int_{A}\theta(s)c(\xi_s,a)\pi^{*}(da|s,\xi_s)ds\biggr)\biggr]\nonumber\\ 
&\geq E^{\pi^{*}}_x\biggl[\liminf_{t\rightarrow\infty}exp\biggl(\int_{0}^{t}\int_{A}\theta(s)c(\xi_s,a)\pi^{*}(da|s,\xi_s)ds\biggr)\biggr]\nonumber\\
&=\tilde{J}_\alpha(\theta,x,\pi^{*}).\label{eq 3.16}
 \end{align}
 From (\ref{eq 3.15}) and  (\ref{eq 3.16}), we have
 \begin{align*}
  \tilde{J}_\alpha(\theta,x,\pi^{*})\leq \varphi_\alpha(\theta,x).
  \end{align*}
   Thus
 \begin{align}
 \inf_{\pi\in\Pi} \tilde{J}_\alpha(\theta,x,\pi)\leq  \tilde{J}_\alpha(\theta,x,\pi^{*})\leq \varphi_\alpha(\theta,x).\label{eq 3.17}
 \end{align}
From (\ref{eq 3.14}) and (\ref{eq 3.17}), we have (\ref{eq 3.2}).
 \end{proof}

\section{The existence of solution to the HJB equation}\label{EHJB}
In this Section, we prove that the equation (\ref{eq 3.1}) is the HJB equation for the  $\alpha$ discounted cost  (\ref{eq 2.6}) and the equation (\ref{eq 3.1}) has a solution in $ B^1_{V_0,V_1}([0,1]\times S)$. We now proceed to make a rigorous analysis of the above. First we truncate our transition and cost rates
which plays a crucial role to derive the HJB equations and find the solution. Fix any $n\geq 1$, $0<\delta<1$. For each $n\geq 1$, $x\in S$, $a\in A(x)$, let $A_n(x):=A(x)$, $S_n:=\{x\in S|V_0(x)\leq n\}$, and $K_n:=\{(x,a)|x\in S_n,a\in A_n(x)\}$. Moreover for each $x\in S$, $a\in A_n(x)$
	define 	\begin{align}
q^{(n)}(dy|x,a)	:=\left\{ \begin{array}{ll}	
	q(dy|x,a)~\text{ if}~x\in S_n,\\
	0~\text{ if }~x\notin S_n\label{eq 4.1}
	\end{array}\right.
	\end{align}
	and 
		\begin{align}
	c_n(x,a):=\left\{ \begin{array}{ll}	
	c(x,a)\wedge \text{ min}\{n,\rho_1\ln V_0(x)+L_0\}~\text{ if}~x\in S_n,\\
	0~\text{ if }~x\notin S_n.\label{eq 4.2}
	\end{array}\right.
	\end{align}
\begin{lemma}\label{lemm 4.1}
	Suppose Assumptions \ref{assm 2.1}, \ref{assm 2.2} and \ref{assm 3.1} are satisfied. 
Then, there exists a unique function $\varphi^{(n,\delta)}_\alpha$  (depending on $n$, $\delta$) in $B^1_{V_0,V_1}([0,1]\times S)$ for  which the followings are true :
	\begin{enumerate}
\item  		 $\varphi^{(n,\delta)}_\alpha\in B_{1}([0,1]\times S)$ is a bounded solution to the following  partial differential equations (PDEs) for all $x\in S$ and a.e. $\theta\in (\delta,1]:$
\interdisplaylinepenalty=0
	\begin{align}
\left\{ \begin{array}{ll}	\alpha\theta\frac{\partial\varphi^{(n,\delta)}_\alpha}{\partial\theta}(\theta,x)
&=\displaystyle{\inf_{a\in  A(x)}\biggl[\theta c_n(x,a)\varphi^{(n,\delta)}_\alpha(\theta,x)+\int_{S}q^{(n)}(dy|x,a)\varphi^{(n,\delta)}_\alpha(\theta,y)\biggr]}\\
 \varphi^{(n,\delta)}_\alpha(\delta,x)&=e^{{n\delta}/{\alpha}}.\label{eq 4.3}
\end{array}\right.
\end{align}
		\item $\varphi^{(n,\delta)}_\alpha(\theta,x)$ has a stochastic representation as follows: for each $x\in S$ and a.e. $\theta\in (\delta,1]$,
		\interdisplaylinepenalty=0
\begin{align}
	\varphi_\alpha^{(n,\delta)}(\theta,x)\nonumber
&=\inf_{\pi\in \Pi}E^{\pi}_x\biggl[e^{{n\delta}/{\alpha}}exp\biggl(\theta\int_{0}^{T_\delta(\theta)}\int_{A}e^{-\alpha t}c_n(\xi^{(n)}_t,a)\pi(da|\omega,t)dt\biggr)\biggr],\nonumber\\\label{eq 4.4}
\end{align}
where $T_\delta(\theta):=\alpha^{-1}\log(\theta/ \delta)$ and $\xi^{(n)}_t$ is the process corresponding to the $q^{(n)}(\cdot|x,a)$.
		\end{enumerate}
	\end{lemma}
\begin{proof} (1) Since $S_n:=\{x\in S|V_0(x)\leq n\}$, by Assumption \ref{assm 2.1}(ii), we say that $q^{(n)}_x(a):=\int_{S/\{x\}}q^{(n)}(dy|x,a)$ is bounded. So we can use the Lyapunov function $V\equiv 1$ such that $\int_{S}q^{(n)}(dy|x,a)V(y)\leq \rho_0 V(x)$, and $\overline{q}^{(n)}:=\sup_{(x,a)\in K}q^{(n)}_x(a)<\infty$.
		Now let us define an nonlinear operator $T$ on $B_{1}([0,1]\times S)$ as follows:
		\begin{align*}
		Tu(\theta,x)=&e^{{\delta n}/{\alpha}}+\frac{1}{\alpha}\int_{\delta}^{\theta}\inf_{a\in A(x)}\biggl[\frac{1}{s}\int_{S}q^{(n)}(dy|x,a)u(s,y)+c_n(x,a)u(s,x)\biggr]ds,
		\end{align*}
		where $u\in B_{1}([0,1]\times S)$ and $(\theta,x)\in [\delta,1]\times S$.	
		By using the Assumption \ref{assm 2.1} and the fact that $c_n$ is bounded, we obtain
		\interdisplaylinepenalty=0
		\begin{align*}
		&\sup_{\theta\in [\delta,1]}\sup_{x\in S}|Tu(\theta,x)|\\
		&\leq e^{{\delta n}/{\alpha}}+\frac{1}{\alpha}\int_{\delta}^{1}\sup_{a\in A(x)}\biggl\{\frac{1}{s}\sup_{x\in S}\biggl[\int_{S}|q^{(n)}(dy|x,a)||u(s,y)|\biggr]+n\sup_{x\in S}|u(s,x)|\biggr\}ds\\
		&\leq e^{{\delta n}/{\alpha}}+\frac{\|u\|_1^\infty}{\alpha}\biggl\{\int_{\delta}^{1}\sup_{a\in A(x)}\frac{1}{s}\sup_{x\in S}\biggl(2q^{(n)}_x(a)\biggr)ds+n(1-\delta)\biggr\}\\
		&\leq e^{{\delta n}/{\alpha}}+\frac{1}{\alpha}\biggl[(-2)\overline{q}^{(n)}\log{\delta}+n(1-\delta)\biggr]{\|u\|_1^\infty}.
		\end{align*}
		Therefore, $T$ is a nonlinear operator from $B_{1}([0,1]\times S)$ to $B_{1}([0,1]\times S)$. For any $g_1,g_2\in B_{1}([0,1]\times S)$ and $\theta\in [\delta,1]$, we have
		\interdisplaylinepenalty=0
		\begin{align}
		\sup_{x\in S}|Tg_1(t,x)-Tg_2(t,x)|
		&\leq \frac{1}{\alpha}\int_{\delta}^{t}\biggl(2\overline{q}^{(n)}/s+n\biggr)\sup_{x\in S}|g_1(s,x)-g_2(s,x)|ds\nonumber\\
		&\leq \frac{1}{\alpha}\biggl[2\overline{q}^{(n)}(\log t-\log \delta)+n(t-\delta)\biggr]\|g_1-g_2\|^\infty_1.\label{eq 4.5}
		\end{align}
	Now, we prove the following:
		\begin{equation}
		\sup_{x\in S}|T^lg_1(t,x)-T^lg_2(t,x)|\leq \frac{\|g_1-g_2\|^\infty_1}{\alpha^l\cdot  l!}\biggl[2\overline{q}^{(n)}(\log t-\log \delta)+n(t-\delta)\biggr]^l~~\forall~l\geq1.\label{eq 4.6}
		\end{equation}
		By (\ref{eq 4.5}) and (\ref{eq 4.6}) we have
		\interdisplaylinepenalty=0
		\begin{align*}
		&\sup_{x\in S}|T^{l+1}g_1(t,x)-T^{l+1}g_2(t,x)|\\
		&\leq \frac{1}{\alpha}\int_{\delta}^{t}\biggl(2\overline{q}^{(n)}/s+n\biggr)\sup_{x\in S}|T^lg_1(s,x)-T^lg_2(s,x)|ds\\
		&	\leq  \frac{\|g_1-g_2\|^\infty_1}{\alpha^{l+1}\cdot  l!}\int_{\delta}^{t}\biggl(2\overline{q}^{(n)}/s+n\biggr)\biggl[2\overline{q}^{(n)}(\log s-\log \delta)+n(s-\delta)\biggr]^l ds\\
		&=\frac{\|g_1-g_2\|^\infty_1}{\alpha^{l+1} \cdot (l+1)!}\biggl[2\overline{q}^{(n)}(\log t-\log \delta)+n(t-\delta)\biggr]^{l+1}.
		\end{align*}
		Since $\sum_{k\geq 1}\frac{1}{\alpha^k \cdot k!}\biggl[-2\overline{q}^{(n)}\log \delta+n(1-\delta)\biggr]^k<\infty$, there exists some $m$ such that 
		$\beta:=\frac{1}{\alpha^m \cdot m!}\biggl[-2\overline{q}^{(n)}\log \delta+n(1-\delta)\biggr]^m<1,$
		which implies that $\|T^m g_1-T^m g_2\|_1^\infty\leq 
		\beta \|g_1-g_2\|^\infty_1$. Therefore, T is a $m$-step contraction operator on $B_{1}([0,1]\times S)$. So, by Banach fixed point theorem, there exists a unique bounded function $\varphi_\alpha^{(n,\delta)}\in B_{1}([0,1]\times S)$ (depending on $(n,\delta)$) such that $T\varphi^{(n,\delta)}_\alpha=\varphi^{(n,\delta)}_\alpha$; that is,
		\interdisplaylinepenalty=0
		\begin{align*}
		\varphi^{(n,\delta)}_\alpha(\theta,x)
		=e^{{\delta n}/{\alpha}}+\frac{1}{\alpha}\int_{\delta}^{\theta}\inf_{a\in A(x)}\biggl[\frac{1}{s}\int_{S}q^{(n)}(dy|x,a)\varphi^{(n,\delta)}_\alpha(s,y)+c_n(x,a)\varphi^{(n,\delta)}_\alpha(s,x)\biggr]ds.
		\end{align*}
		Also note that $\varphi^{(n,\delta)}_\alpha(\delta,x)=e^{{\delta n}/{\alpha}}$.
		Hence by using (\ref{eq 4.1}), (\ref{eq 4.2}) and the above equation, we have   $\varphi^{(n,\delta)}_\alpha\in  B^1_{V_0,V_1}([0,1]\times S)$ and it satisfies equation (\ref{eq 4.3}).

(2) We see that
\begin{align*}
\biggl[\theta c_n(x,a)\varphi^{(n,\delta)}_\alpha(\theta,x)+\int_{S}q^{(n)}(dy|x,a)\varphi^{(n,\delta)}_\alpha(\theta,y)\biggr] 
\end{align*}
is continuous in $a\in A(x)$ and $A(x)$ is compact. So by measurable selection theorem, [\cite{BS},Proposition 7.33], there exists a measurable function $f_\delta^{*}:[0,1]\times S\rightarrow A$ such that 
\begin{align}
&\inf_{a\in A(x)}\biggl[\theta c_n(x,a)\varphi^{(n,\delta)}_\alpha(\theta,x)+\int_{S}q^{(n)}(dy|x,a)\varphi^{(n,\delta)}_\alpha(\theta,y)\biggr]\nonumber\\
&=\biggl[\theta c_n(x,f_\delta^{*}(\theta,x))\varphi^{(n,\delta)}_\alpha(\theta,x)+\int_{S}q^{(n)}(dy|x,f_\delta^{*}(\theta,x))\varphi^{(n,\delta)}_\alpha(\theta,y)\biggr].\label{eq 4.7}
\end{align}
Let
\begin{equation*}
\pi_\delta^{*}:\mathbb{R}_+\times S \to P(A)  \
\end{equation*}
be defined by
	\begin{align*}
&\pi_\delta^{*}(\cdot|t,x)=I_{\{\hat{f}_\delta^{*}(t,x)\}}(\cdot),~\text{where}~\hat{f}_\delta^{*}:[0,1]\times S\rightarrow A,\\&\quad~\text{be a measurable mapping, defined by}~ \hat{f}_\delta^{*}(t,x):=f_\delta^{*}(\theta e^{-\alpha t},x).
\end{align*}
  Let $\theta(t):=\theta e^{-\alpha t}$ for $t\in [0,\infty)$. Since $c_n$ and $\varphi_\alpha^{(n,\delta,k)}$ are bounded, by Dynkin's formula we get
	 \interdisplaylinepenalty=0
	\begin{align}
	&E^{\pi}_x\biggl[exp\biggl(\int_{0}^{T_\delta(\theta)}\int_{A}\theta(s)c_n(\xi^{(n)}_s,a)\pi(da|\omega,s)ds\biggr)\varphi^{(n,\delta)}_\alpha\biggl(\theta(T_\delta),\xi_{T_\delta}^{(n)}\biggr)\biggr]-\varphi^{(n,\delta)}_\alpha(\theta,x)\nonumber\\
	&=E^{\pi}_x\biggl\{\int_{0}^{T_\delta(\theta)}\biggl[-\alpha \theta(s)\frac{\partial\varphi^{(n,\delta)}_\alpha}{\partial\theta}(\theta(s),\xi^{(n)}_s)
	+\int_{S}\int_{A}q^{(n)}(dy|\xi^{(n)}_s,a)\varphi^{(n,\delta)}_\alpha(\theta(s),y)\pi(da|\omega,s)\nonumber\\
	&+\theta(s)\int_{A}c_n(\xi^{(n)}_s,a)\varphi_\alpha^{(n,\delta)}(\theta(s),\xi^{(n)}_s)\pi(da|\omega,s)\biggr]\nonumber\\
	&exp\biggl(\int_{0}^{s}\int_{A}\theta(v)c_n(\xi_v^{(n)},a)\pi(da|\omega,v)dv\biggr)ds\biggr\}.\label{eq 4.8}
		\end{align}
By using (\ref{eq 4.3}) and (\ref{eq 4.8}), we obtain
\interdisplaylinepenalty=0 
	\begin{align*}
	& E^{\pi}_x\biggl[exp\biggl(\int_{0}^{T_\delta(\theta)}\int_{A}\theta(s)c_n(\xi^{(n)}_s,a)\pi(da|\omega,s)ds\biggr)\varphi^{(n,\delta)}_\alpha\biggl(\theta(T_\delta),\xi_{T_\delta}^{(n)}\biggr)\biggr]\geq \varphi^{(n,\delta)}_\alpha(\theta,x).
	\end{align*}
		Since $\pi\in \Pi$ is an arbitrary control and $\varphi_\alpha^{(n,\delta)}(\theta(T_\delta(\theta)),\xi^{(n)}_{T_\delta})=e^{{n \delta}/{\alpha}}$, we have
		\interdisplaylinepenalty=0
		\begin{align}
	 &\varphi^{(n,\delta)}_\alpha(\theta,x)\leq \inf_{\pi\in\Pi}E^{\pi}_x\biggl[e^{{n \delta}/{\alpha}}exp\biggl(\int_{0}^{T_\delta(\theta)}\int_{A}\theta(s)c_n(\xi^{(n)}_s,a)\pi(da|\omega,s)ds\biggr)\biggr]. \label{eq 4.9}
		\end{align}
Using equations (\ref{eq 4.3}), (\ref{eq 4.7}) and (\ref{eq 4.8}), we can show that
\interdisplaylinepenalty=0
	\begin{align*}
	&\varphi^{(n,\delta)}_\alpha(\theta,x)=E^{\pi_\delta^{*}}_x\biggl[e^{{n \delta}/{\alpha}}exp\biggl(\int_{0}^{T_\delta(\theta)}\int_{A}\theta(s)c_n(\xi^{(n)}_s,a)\pi_\delta^{*}(da|s,\xi^{(n)}_s)ds\biggr)\biggr].
	\end{align*}
Therefore	
\begin{align}
&\varphi^{(n,\delta)}_\alpha(\theta,x) \geq \inf_{\pi\in \Pi} E^{\pi}_x\biggl[e^{{n \delta}/{\alpha}}exp\biggl(\int_{0}^{T_\delta(\theta)}\int_{A}\theta(s)c_n(\xi^{(n)}_s,a)\pi(da|\omega,s)ds\biggr)\biggr].\label{eq 4.10}
\end{align}	
Therefore, from (\ref{eq 4.9}) and (\ref{eq 4.10}), we obtain (\ref{eq 4.4}). This completes the proof.
		\end{proof}
\begin{thm}\label{theo 4.1}
Suppose Assumptions \ref{assm 2.1}, \ref{assm 2.2} and \ref{assm 3.1} hold. Then the HJB equation (\ref{eq 3.1}) has a unique solution $\varphi_\alpha\in B^1_{V_0,V_1}([0,1]\times S)$ satisfying  
$1\leq \varphi_\alpha(\theta,x)\leq {\frac{\alpha^2 e^{{\theta L_0}/{\alpha}}}{\alpha^2-\rho_0\rho_1\theta}} (V(x))^{\frac{\rho_1\theta}{\alpha}}$ for all $(\theta,x)\in [0,1]\times S.$
	\end{thm}
 \begin{proof}
 	First note that, $\varphi^{(n,\delta)}_\alpha$ is the solution to the equation (\ref{eq 4.3}), which depends on two parameters $n$, $\delta$. We prove this theorem in two steps. 

 	 Step 1.  In the first step, we construct a solution  $\varphi^{(n)}_\alpha(\cdot,x)$ from $\varphi^{(n,\delta)}_\alpha(\cdot,x)$  by passing the limit as $\delta \rightarrow 0$, such that $\varphi^{(n)}_\alpha(\cdot,x)$ is absolutely continuous function and satisfies the following DEs:
 	 \interdisplaylinepenalty=0
 	 	\begin{align}
 	 \left\{ \begin{array}{lllll}\alpha\theta\frac{\partial\varphi^{(n)}_\alpha}{\partial \theta}(\theta,x)
 	 &=\displaystyle{\inf_{a\in A(x)}}\biggl[\int_{S}q^{(n)}(dy|x,a)\varphi_\alpha^{(n)}(\theta,y)+\theta c_n(x,a)\varphi^{(n)}_\alpha(\theta,x)\biggr],~x\in S,~\text{a.e.},~\theta\in [0,1],\\
 	 1\leq \varphi^{(n)}_\alpha(\theta,x)&\leq {\frac{\alpha^2 e^{{\theta L_0}/{\alpha}}}{\alpha^2-\rho_0\rho_1\theta}}(V_0(x))^{\frac{\rho_1\theta}{\alpha}}~~~ \forall~(\theta,x)\in [0,1]\times S.\label{eq 4.11}
 	 \end{array}\right.
 	 \end{align}	  
 	  Given $0<\delta<1$ and $1\leq n<\infty$  by (\ref{eq 4.4}) and $\sup_{(x,a)\in K}c_n(x,a)\leq n$, we have 
 	  \interdisplaylinepenalty=0
 	 $\varphi^{(n,\delta)}_\alpha(\theta,x)\leq e^{2n/\alpha},~~x\in S,\theta\in [\delta,1]$.
 	 
 	 Next, we extend the domain of $\varphi^{(n,\delta)}_\alpha$ to $[0,1]\times S$ by 	 
 	 $$\overline{\varphi}^{(n,\delta)}_{\alpha}(\theta,x)=\left\{ \begin{array}{lcl}{\varphi}^{(n,\delta)}_{\alpha}(\theta,x) ,  &\delta\leq \theta\leq 1~\forall x\in S\\ e^{{n\delta}/{\alpha}},  & 0\leq \theta<\delta~\forall x\in S.
 	 \end{array}\right.$$
 We consider the following expression, for any given $\pi\in \Pi$, $x\in S$, $\theta,\theta_0\in [\delta,1]$:
 \interdisplaylinepenalty=0
 	 \begin{align*}
 		 \biggl|E^{\pi}_x &\biggl[e^{n\delta/\alpha}exp\biggl(\theta\int_{0}^{T_\delta(\theta)}\int_{A} e^{-\alpha t} c_n(\xi^{(n)}_t,a)\pi(da|\omega,t)dt\biggr)\biggr]\\
 	&-E^{\pi}_x \biggl[e^{n\delta/\alpha}exp\biggl(\theta_0\int_{0}^{T_\delta(\theta_0)}\int_{A}e^{-\alpha t} c_n(\xi^{(n)}_t,a)\pi(da|\omega,t)dt\biggr)\biggr]\biggr|\\
 	&\leq P_1+P_2,
 	 \end{align*}
   where
   \interdisplaylinepenalty=0
 	 \begin{align*}
  P_1:=&\biggl|E^{\pi}_x \biggl[e^{n\delta/\alpha}exp\biggl(\theta\int_{0}^{T_\delta(\theta)}\int_{A} e^{-\alpha t} c_n(\xi^{(n)}_t,a)\pi(da|\omega,t)dt\biggr)\biggr]\\
 &-E^{\pi}_x \biggl[e^{n\delta/\alpha}exp\biggl(\theta_0\int_{0}^{T_\delta(\theta)}\int_{A} e^{-\alpha t} c_n(\xi^{(n)}_t,a)\pi(da|\omega,t)dt\biggr)\biggr]\biggr|,
 	 \end{align*}
 	 and 
 	 \interdisplaylinepenalty=0
 	 \begin{align*}
 	 P_2:=&\biggl|E^{\pi}_x \biggl[e^{n\delta/\alpha}exp\biggl(\theta_0\int_{0}^{T_\delta(\theta)}\int_{A} e^{-\alpha t} c_n(\xi^{(n)}_t,a)\pi(da|\omega,t)dt\biggr)\biggr]\\
 	&-E^{\pi}_x \biggl[e^{n\delta/\alpha}exp\biggl(\theta_0\int_{0}^{T_\delta(\theta_0)}\int_{A} e^{-\alpha t} c_n(\xi^{(n)}_t,a)\pi(da|\omega,t)dt\biggr)\biggr]\biggr|.
 	 \end{align*}
Fix $n\geq 1$; we have
 	\begin{align*}
  	\int_{0}^{T_\delta(\theta)}\int_{A} e^{-\alpha t} c_n(\xi^{(n)}_t,a)\pi(da|\omega,t)dt\leq n\int_{0}^{T_\delta(\theta)}e^{-\alpha t}dt\leq \frac{n}{\alpha} 
 	\end{align*}
 	 and 
 	 \interdisplaylinepenalty=0
 \begin{align*}
 \int_{T_\delta(\theta\wedge\theta_0)}^{T_\delta(\theta\vee\theta_0)}&\int_{A} e^{-\alpha t} c_n(\xi^{(n)}_t,a)\pi(da|\omega,t)dt
\leq n \int_{T_\delta(\theta\wedge\theta_0)}^{T_\delta(\theta\vee\theta_0)}e^{-\alpha t}dt\\
&=\frac{n}{\alpha}[exp(-\alpha T_\delta(\theta\wedge \theta_0))-exp(-\alpha T_\delta(\theta\vee \theta_0))]\leq \frac{\delta n |\theta_0-\theta|}{\alpha \theta \theta_0},
 \end{align*}
where $c\wedge d:=min\{c,d\}$~~and $c\vee d:=max\{c,d\}$.	
 Hence, we obtain
 \interdisplaylinepenalty=0
 \begin{align*}
 P_1&=e^{n\delta/\alpha}E^{\pi}_x \biggl[exp\biggl((\theta\wedge\theta_0)\int_{0}^{T_\delta(\theta)}\int_{A} e^{-\alpha t} c_n(\xi^{(n)}_t,a)\pi(da|\omega,t)dt\biggr)\\
&~~~~\times
\biggl(exp\biggl(|\theta_0-\theta|\int_{0}^{T_\delta(\theta)}\int_{A} e^{-\alpha t} c_n(\xi^{(n)}_t,a)\pi(da|\omega,t)dt\biggr)-1\biggr)\biggr]\\
&\leq e^{2 n/\alpha}E^{\pi}_x\biggl[exp\biggl(|\theta_0-\theta|\int_{0}^{T_\delta(\theta)}\int_{A} e^{-\alpha t} c_n(\xi^{(n)}_t,a)\pi(da|\omega,t)dt\biggr)-1\biggr]\\
&\leq e^{2 n/\alpha} \Big(exp\Big(\frac{n}{\alpha}|\theta_0-\theta|\Big)-1\Big)\\
&\leq e^{2 n/\alpha}\Big(e^{n/\alpha}-1\Big)|\theta_0-\theta|.
 \end{align*}
 Here, the last inequality follows from the fact that $e^{bz}-1\leq (e^b-1)z$ for all $z\in [0,1]$ and $b>0$. 	 
 	 Similarly for $P_2$ we have
 	 \interdisplaylinepenalty=0
 	 \begin{align*}
 P_2&=e^{n\delta/\alpha}E^{\pi}_x \biggl[exp\biggl(\theta_0\int_{0}^{T_\delta(\theta\wedge\theta_0)}\int_{A} e^{-\alpha t} c_n(\xi^{(n)}_t,a)\pi(da|\omega,t)dt\biggr)\\
  &~~~~\times\biggl(exp\biggl(\theta_0\int_{T_\delta(\theta\wedge\theta_0)}^{T_\delta(\theta\vee\theta_0)}\int_{A} e^{-\alpha t} c_n(\xi^{(n)}_t,a)\pi(da|\omega,t)dt\biggr)-1\biggr)\biggr]\\
&\leq e^{2 n/\alpha}E^{\pi}_i \biggl[exp\biggl(\theta_0\int_{T_\delta(\theta\wedge\theta_0)}^{T_\delta(\theta\vee\theta_0)}\int_{A} e^{-\alpha t} c_n(\xi^{(n)}_t,a)\pi(da|\omega,t)dt\biggr)-1\biggr]\\
&\leq e^{2 n/\alpha}\biggl(exp\biggl(\frac{n\delta|\theta-\theta_0|}{\alpha \theta}\biggr)-1\biggr)\\
&\leq e^{2 n/\alpha}\biggl(e^{n/\alpha}-1\biggr)|\theta_0-\theta|.
 	 \end{align*}
 		Hence for all $(\theta,x)\in [0,1]\times S$, we have
 	\begin{align}
 	|\overline{\varphi}^{(n,\delta)}_\alpha(\theta_0,x)-\overline{\varphi}^{(n,\delta)}_\alpha(\theta,x)|
 \leq 2e^{{2n}/{\alpha}}(e^{{n}/{\alpha}}-1)|\theta-\theta_0|.\label{eq 4.12}
 	\end{align}
 	Now we want to show that $\overline{\varphi}^{(n,\delta)}_\alpha$ is decreasing as $\delta\rightarrow 0$ for any $(\theta,x)$. For a fixed $\alpha>0$ and $\varepsilon>0$ small enough, consider
 	$\overline{\varphi}^{(n,\delta+\varepsilon)}_\alpha(\theta,x)-\overline{\varphi}^{(n,\delta)}_\alpha(\theta,x)$ and assume that $h_\delta:=e^{\frac{n\delta}{\alpha}}$. By measurable selection theorem we get the minimizer $\pi^{*}_{\delta+\varepsilon}$ like in equation (\ref{eq 4.7}), corresponding to $\overline{\varphi}_\alpha^{(n,\delta+\varepsilon)}$ such that the followings cases hold.\\
 	Case 1. If $\delta+\varepsilon<\theta$ then
 	\begin{align*}
&\overline{\varphi}^{(n,\delta+\varepsilon)}_\alpha(\theta,x)-\overline{\varphi}^{(n,\delta)}_\alpha(\theta,x)\\
&=E^{\pi^{*}_{\delta+\varepsilon}}_x\biggl[h_{\delta+\varepsilon} exp\biggl(\theta\int_{0}^{T_{\delta+\varepsilon}} e^{-\alpha t} c_n(\xi^{(n)}_t,\pi^{*}_{\delta+\varepsilon}(da|t,\xi^{(n)}_t))dt\biggr)\biggr]\\
&\quad-\inf_{\pi\in \Pi}E^{\pi}_x\biggl[h_{\delta} exp\biggl(\theta\int_{0}^{T_\delta} e^{-\alpha t} c_n(\xi^{(n)}_t,\pi(da|\omega,t))dt\biggr)\biggr]\\
&\geq
h_\delta E^{\pi^{*}_{\delta+\varepsilon}}_x\biggl[ exp\biggl(\theta\int_{0}^{T_{\delta+\varepsilon}} e^{-\alpha t} c_n(\xi^{(n)}_t,\pi^{*}_{\delta+\varepsilon}(da|t,\xi^{(n)}_t))dt\biggr)\\
&\quad\times\biggl\{h_\varepsilon-exp\biggl(\theta\int_{T_{\delta+\varepsilon}}^{T_\delta} e^{-\alpha t} c_n(\xi^{(n)}_t,\pi^{*}_{\delta+\varepsilon}(da|t,\xi^{(n)}_t))dt\biggr)\biggr\}\biggr]\\
&\geq h_\delta E^{\pi^{*}_{\delta+\varepsilon}}_x\biggl[ exp\biggl(\theta\int_{0}^{T_{\delta+\varepsilon}} e^{-\alpha t} c_n(\xi^{(n)}_t,\pi^{*}_{\delta+\varepsilon}(da|t,\xi^{(n)}_t))dt\biggr)\biggl\{h_\varepsilon-exp\biggl(\theta\int_{T_{\delta+\varepsilon}}^{T_\delta} e^{-\alpha t} ndt\biggr)\biggr\}\biggr]\\
&=h_\delta E^{\pi^{*}_{\delta+\varepsilon}}_x\biggl[ exp\biggl(\theta\int_{0}^{T_{\delta+\varepsilon}} e^{-\alpha t} c_n(\xi^{(n)}_t,\pi^{*}_{\delta+\varepsilon}(da|t,\xi^{(n)}_t))dt\biggr)\biggl\{h_\varepsilon-exp\biggl(\frac{n\theta(e^{-\alpha T_{\delta+\varepsilon}}-e^{-\alpha T_\delta})}{\alpha}\biggr)\biggr\}\biggr]\\
&=0.
 	\end{align*}
 	Case 2. $\delta<\theta\leq \delta+\varepsilon$ 
 	 	\begin{align*}
 	&\overline{\varphi}^{(n,\delta+\varepsilon)}_\alpha(\theta,x)-\overline{\varphi}^{(n,\delta)}_\alpha(\theta,x)\\
 	&=h_{\delta+\varepsilon}-E^{\pi^{*}_{\delta}}_x\biggl[h_{\delta} exp\biggl(\theta\int_{0}^{T_\delta} e^{-\alpha t} c_n(\xi^{(n)}_t,\pi^{*}_\delta(da|t,\xi^{(n)}_t))dt\biggr)\biggr]\\
 	&=h_\delta\biggl [h_\varepsilon-E^{\pi^{*}_{\delta}}_x\biggl[exp\biggl(\theta\int_{{0}}^{T_\delta} e^{-\alpha t} c_n(\xi^{(n)}_t,\pi^{*}_{\delta}(da|t,\xi^{(n)}_t))dt\biggr)\biggr]\biggr]\\
 		&\geq h_\delta\biggl [h_\varepsilon-exp\biggl(\theta\int_{{0}}^{T_\delta} e^{-\alpha t} ndt\biggr)\biggr]\\
 		&= h_\delta\biggl [h_\varepsilon-e^{n\theta \frac{(1-e^{-\alpha T_\delta})}{\alpha}}\biggr]\geq 0.
 	\end{align*}
 	Case 3. $\theta\leq \delta$
 	$$\overline{\varphi}^{(n,\delta+\varepsilon)}_\alpha(\theta,x)-\overline{\varphi}^{(n,\delta)}_\alpha(\theta,x)=h_{\delta+\varepsilon}-h_\delta=h_\delta(h_\varepsilon-1)=h_\delta(e^\frac{n\varepsilon}{\alpha}-1)\geq 0.$$
 		Hence $\overline{\varphi}_\alpha^{(n,\delta)}(\theta,x)$ is increasing in $\delta$ for any $(\theta,x)\in [0,1]\times S$. Now from (\ref{eq 4.12}), we know that for each $x\in S$, $\overline{\varphi}_\alpha^{(n,\delta)}(\cdot,x)$ is Lipschitz continous in $\theta\in [0, 1]$. Also, $\overline{\varphi}_\alpha^{(n,\delta)}(\theta,x)$ is increasing in $\delta$ for any $(\theta,x)\in [0,1]\times S$ and bounded above (since $\overline{\varphi}^{(n,\delta)}_\alpha(\theta,x)\leq e^{2n/\alpha},~~x\in S,\theta\in [\delta,1]$), therefore there exists a function $\varphi^{(n)}_\alpha$ on $[0,1]\times S$ that is continuous with respect to $\theta\in [0,1]$, such that along a subsequence $\delta_m\rightarrow 0$, we have $\lim_{m\rightarrow\infty}\overline{\varphi}_\alpha^{(n,\delta_m)}(\theta,x)=\varphi_\alpha^{(n)}(\theta,x)$ and for any fixed $x\in S$  this convergence is uniform in $\theta\in [0,1]$.
 
 		Let $\psi\in C^\infty_c(0,1)$, then we have
 		\begin{align}
 	&-\int_{0}^{1}\alpha\frac{d(\theta\psi)}{d\theta}(\theta)\overline{\varphi}_\alpha^{(n,\delta_m)}(\theta,x)d\theta=\int_{0}^{1}\alpha\theta\frac{\partial \overline{\varphi}^{(n,\delta_m)}_\alpha}{\partial \theta}(\theta,x)\psi(\theta)d\theta\nonumber\\
 	&=\int_{0}^{1}\displaystyle{\inf_{a\in A(x)}}\biggl[\theta c_n(x,a)\overline{\varphi}^{(n,\delta_m)}_\alpha(\theta,x)+\int_{S}q^{(n)}(dy|x,a)\overline{\varphi}^{(n,\delta_m)}_\alpha(\theta,y)\biggr]\psi(\theta)d\theta\nonumber\\
 	&-\int_{0}^{\delta_m}\inf_{a\in A(x)}\biggl[\theta c_n(x,a)\overline{\varphi}^{(n,\delta_m)}_\alpha(\theta,x)+\int_{S}q^{(n)}(dy|x,a)\overline{\varphi}^{(n,\delta_m)}_\alpha(\theta,y)\biggr]\psi(\theta)d\theta\nonumber\\
 		&=\int_{0}^{1}\displaystyle{\inf_{a\in  A(x)}\biggl[\theta c_n(x,a)\overline{\varphi}^{(n,\delta_m)}_\alpha(\theta,x)+\int_{S}q^{(n)}(dy|x,a)\overline{\varphi}^{(n,\delta_m)}_\alpha(\theta,y)\biggr]}\psi(\theta)d\theta\nonumber\\
 	&-\int_{0}^{\delta_m}\inf_{a\in A(x)}\biggl[\theta c_n(x,a)\overline{\varphi}^{(n,\delta_m)}_\alpha(\theta,x)\biggr]\psi(\theta)d\theta.\label{eq 4.13}
 		\end{align}
 		Now take $\tau(x):=M_0 V_0(x)$ and define $$Q^{(n)}(dy|x,a):=\delta_x(dy)+\frac{q^{(n)}(dy|x,a)}{\tau(x)}$$ for all $(x,a)\in K$ where $\delta_x(\cdot)$ is the Dirac measure concentrated at $x$. We see that under Assumption \ref{assm 2.1}, $Q^{(n)}$ is a stochatic kernel on $S$ given $K$.
 	Then (\ref{eq 4.13}) can be written as 
 		\begin{align}
 	&-\int_{0}^{1}\biggl\{\frac{\alpha}{\tau(x)}\frac{d(\theta\psi)}{d\theta}\overline{\varphi}_\alpha^{(n,\delta_m)}(\theta,x)-\overline{\varphi}_\alpha^{(n,\delta_m)}(\theta,x)\psi(\theta)\biggr\}d\theta\nonumber\\
 	&=\int_{0}^{1}\displaystyle{\inf_{a\in A(x)}}\biggl[\frac{\theta}{\tau(x)} c_n(x,a)\overline{\varphi}^{(n,\delta_m)}_\alpha(\theta,x)+\int_{S}Q^{(n)}(dy|x,a)\overline{\varphi}^{(n,\delta_m)}_\alpha(\theta,y)\biggr]\psi(\theta)d\theta\nonumber\\
 	&-\frac{1}{\tau(x)}\int_{0}^{\delta_m}\inf_{a\in A(x)}\biggl[\theta c_n(x,a)\overline{\varphi}^{(n,\delta_m)}_\alpha(\theta,x)\biggr]\psi(\theta)d\theta.\label{eq 4.14}
 	\end{align}
 	
 		Now 	\begin{align}
 	&\biggl|\displaystyle{\inf_{a\in  A(x)}\biggl[\frac{\theta}{\tau(x)} c_n(x,a)\overline{\varphi}^{(n,\delta_m)}_\alpha(\theta,x)+\int_{S}Q^{(n)}(dy|x,a)\overline{\varphi}^{(n,\delta_m)}_\alpha(\theta,y)\biggr]}\psi(\theta)\biggr|\nonumber\\
 		&\leq 	|\psi(\theta)|\displaystyle{\sup_{a\in A(x)}\biggl[\frac{\theta}{\tau(x)} |c_n(x,a)||\overline{\varphi}^{(n,\delta)}_\alpha(\theta,x)|+\int_{S}Q^{(n)}(dy|x,a)|\overline{\varphi}^{(n,\delta)}_\alpha(\theta,y)|\biggr]}\nonumber\\
 			&\leq {\frac{\alpha^2}{\alpha^2-\rho_0\rho_1\theta}} e^{{\theta L_0}/{\alpha}}\displaystyle\sup_{a\in A(x)}\biggl[\frac{\theta}{\tau(x)}nV_0^{\frac{\rho_1\theta}{\alpha}}(x)+\int_{S}Q(dy|x,a)V_0^{\frac{\rho_1\theta}{\alpha}}(y)\biggr]|\psi(\theta)|\nonumber\\
 			&\leq {\frac{\alpha^2}{\alpha^2-\rho_0\rho_1\theta}} e^{{\theta L_0}/{\alpha}}\displaystyle\sup_{a\in  A(x)}\biggl[\frac{\theta}{\tau(x)}nV_0(x)+\int_{S}Q(dy|x,a)V_0(y)\biggr]|\psi(\theta)|\nonumber\\
 				&\leq {\frac{\alpha^2}{\alpha^2-\rho_0\rho_1\theta}} e^{{\theta L_0}/{\alpha}}\displaystyle\biggl[\frac{\theta}{\tau(x)}nV_0(x)+V_0(x)+\rho_0\frac{V_0(x)}{\tau(x)}\biggr]|\psi(\theta)|\nonumber\\
 						&= {\frac{\alpha^2}{\alpha^2-\rho_0\rho_1\theta}} e^{{\theta L_0}/{\alpha}}\displaystyle\biggl[\frac{\theta}{\tau(x)}nV_0(x)+V_0(x)+\frac{\rho_0}{M_0}\biggr]|\psi(\theta)|
 				.\label{eq 4.15}
 			 \end{align} 		
Since for each fixed $x\in S$, $A(x)$ is compact, there exists a subsequence of $\{m\}$, by abuse of notation, we denote the same sequence and  $a^{*}\in A(x)$  such that $\lim_{m\rightarrow\infty}a^{*}_{m}=a^{*}$. 
Now, from (\ref{eq 4.14}), for any $a\in A(x)$, we have
 				\begin{align}
 			&-\int_{0}^{1}\biggl\{\frac{\alpha}{\tau(x)}\frac{d(\theta\psi)}{d\theta}\overline{\varphi}_\alpha^{(n,\delta_m)}(\theta,x)-\overline{\varphi}_\alpha^{(n,\delta_m)}(\theta,x)\psi(\theta)\biggr\}d\theta\nonumber\\
 			&= \int_{0}^{1}\displaystyle\biggl[\frac{\theta}{\tau(x)} c_n(x,a^{*}_m)\overline{\varphi}^{(n,\delta_m)}_\alpha(\theta,x)+\int_{S}Q^{(n)}(dy|x,a^{*}_m)\overline{\varphi}^{(n,\delta_m)}_\alpha(\theta,y)\biggr]\psi(\theta)d\theta\nonumber\\
 			&\quad-\frac{1}{\tau(x)}\int_{0}^{\delta_m}\inf_{a\in A(x)}\biggl[\theta c_n(x,a)\overline{\varphi}^{(n,\delta_m)}_\alpha(\theta,x)\biggr]\psi(\theta)d\theta.\label{eq 4.16}
 			\end{align}
 		So, by Lemma 8.3.7 in Hernandez-Lerma and Lassere (1999) \cite{HL2} taking limit as $m\rightarrow\infty$ in  (\ref{eq 4.16}), we get
 		\begin{align*}
 		&-\int_{0}^{1}\biggl\{\frac{\alpha}{\tau(x)}\frac{d(\theta\psi)}{d\theta}(\theta)\varphi_\alpha^{(n)}(\theta,x)-\varphi_\alpha^{(n)}(\theta,x)\psi(\theta)\biggr\}d\theta\nonumber\\
 		&\geq \int_{0}^{1}\displaystyle\biggl[\frac{\theta}{\tau(x)} c_n(x,a^{*})\varphi^{(n)}_\alpha(\theta,x)+\int_{S}Q^{(n)}(dy|x,a^{*})\varphi^{(n)}_\alpha(\theta,y)\biggr]\psi(\theta)d\theta.
 		\end{align*}
 		Hence
 			\begin{align}
 		&-\int_{0}^{1}\biggl\{\frac{\alpha}{\tau(x)}\frac{d(\theta\psi)}{d\theta}(\theta)\varphi_\alpha^{(n)}(\theta,x)-\varphi_\alpha^{(n)}(\theta,x)\psi(\theta)\biggr\}d\theta\nonumber\\
 		&\geq \inf_{a\in A(x)}\int_{0}^{1}\displaystyle\biggl[\frac{\theta}{\tau(x)} c_n(x,a)\varphi^{(n)}_\alpha(\theta,x)+\int_{S}Q^{(n)}(dy|x,a)\varphi^{(n)}_\alpha(\theta,y)\biggr]\psi(\theta)d\theta.\label{eq 4.17}
 		\end{align}
 		But
 			\begin{align*}
 		&-\int_{0}^{1}\biggl\{\frac{\alpha}{\tau(x)}\frac{d(\theta\psi)}{d\theta}\overline{\varphi}_\alpha^{(n,\delta_m)}(\theta,x)-\overline{\varphi}_\alpha^{(n,\delta_m)}(\theta,x)\psi(\theta)\biggr\}d\theta\nonumber\\
 		&\leq \int_{0}^{1}\displaystyle\biggl[\frac{\theta}{\tau(x)} c_n(x,a)\overline{\varphi}^{(n,\delta_m)}_\alpha(\theta,x)+\int_{S}Q^{(n)}(dy|x,a)\overline{\varphi}^{(n,\delta_m)}_\alpha(\theta,y)\biggr]\psi(\theta)d\theta\nonumber\\
 		&\quad-\frac{1}{\tau(x)}\int_{0}^{\delta_m}\inf_{a\in A(x)}\biggl[\theta c_n(x,a)\overline{\varphi}^{(n,\delta_m)}_\alpha(\theta,x)\biggr]\psi(\theta)d\theta.
 		\end{align*}
 		By similar arguments, we get
 			\begin{align}
 		&-\int_{0}^{1}\biggl\{\frac{\alpha}{\tau(x)}\frac{d(\theta\psi)}{d\theta}(\theta)\varphi_\alpha^{(n)}(\theta,x)-\varphi_\alpha^{(n)}(\theta,x)\psi(\theta)\biggr\}d\theta\nonumber\\
 		&\leq \inf_{a\in A(x)}\int_{0}^{1}\displaystyle\biggl[\frac{\theta}{\tau(x)} c_n(x,a)\varphi^{(n)}_\alpha(\theta,x)+\int_{S}Q^{(n)}(dy|x,a)\varphi^{(n)}_\alpha(\theta,y)\biggr]\psi(\theta)d\theta.\label{eq 4.18}
 		\end{align}
 		From (\ref{eq 4.17}) and (\ref{eq 4.18}), we get
 			\begin{align}
 		&-\int_{0}^{1}\biggl\{\frac{\alpha}{\tau(x)}\frac{d(\theta\psi)}{d\theta}(\theta)\varphi_\alpha^{(n)}(\theta,x)-\varphi_\alpha^{(n)}(\theta,x)\psi(\theta)\biggr\}d\theta\nonumber\\
 		&= \inf_{a\in A(x)}\int_{0}^{1}\displaystyle\biggl[\frac{\theta}{\tau(x)} c_n(x,a)\varphi^{(n)}_\alpha(\theta,x)+\int_{S}Q^{(n)}(dy|x,a)\varphi^{(n)}_\alpha(\theta,y)\biggr]\psi(\theta)d\theta.\label{eq 4.19}
 		\end{align}
 		Thus we obtain
 			\begin{align*}
 		&-\int_{0}^{1}\alpha\frac{d(\theta\psi)}{d\theta}(\theta)\varphi_\alpha^{(n)}(\theta,x)d\theta\nonumber\\
 		&= \inf_{a\in A(x)}\int_{0}^{1}\displaystyle\biggl[{\theta} c_n(x,a)\varphi^{(n)}_\alpha(\theta,x)+\int_{S}q^{(n)}(dy|x,a)\varphi^{(n)}_\alpha(\theta,y)\biggr]\psi(\theta)d\theta.
 		\end{align*}
 		Hence
 			\begin{align*}
 		\alpha\theta\frac{\partial\varphi^{(n)}_\alpha}{\partial\theta}(\theta,x)
 		&= \inf_{a\in A(x)}\displaystyle\biggl[ \theta c_n(x,a)\varphi^{(n)}_\alpha(\theta,x)+\int_{S}q^{(n)}(dy|x,a)\varphi^{(n)}_\alpha(\theta,y)\biggr]~\text{a.e.}~\theta\in [0,1]
 		\end{align*}
 		in the sense of distribution. When $\frac{\partial\varphi^{(n)}_{\alpha}}{\partial\theta}$ does not exist for some $(\theta,x)$, we define
 			\begin{align*}
 		\alpha\theta\frac{\partial\varphi^{(n)}_\alpha}{\partial\theta}(\theta,x)
 		&= \inf_{a\in A(x)}\displaystyle\biggl[ \theta c_n(x,a)\varphi^{(n)}_\alpha(\theta,x)+\int_{S}q^{(n)}(dy|x,a)\varphi^{(n)}_\alpha(\theta,y)\biggr].
 		\end{align*}
 Now for $\theta\in [\delta_m,1]$, by using (\ref{eq 4.4}) and Proposition \ref{prop 2.1}, we have
\begin{align}
&\varphi^{(n,\delta_m)}_\alpha(\theta,x)
=\inf_{\pi\in \Pi} E^{\pi}_x \biggl[e^{n\delta_m/\alpha}\exp\biggl(\theta\int_{0}^{T_{\delta_m}(\theta)}\int_{A}e^{-\alpha t} c_n(\xi^{(n)}_t,a)\pi(da|\omega,t)dt\biggr)\biggr]\nonumber\\
&\leq e^{n\delta_m/\alpha}\inf_{\pi\in \Pi}E^{\pi}_x \biggl[exp\biggl(\theta\int_{0}^{\infty}\int_{B}\int_{A}e^{-\alpha t} c_n(\xi^{(n)}_t,a)\pi(da|\omega,t)dt\biggr)\biggr]\nonumber\\
&\leq e^{n\delta_m/\alpha}\inf_{\pi\in \Pi}E^{\pi}_x \biggl[exp\biggl(\theta\int_{0}^{\infty}\int_{B}\int_{A}e^{-\alpha t} c(\xi^{(n)}_t,a)\pi(da|\omega,t)dt\biggr)\biggr]\nonumber\\
&\leq e^{n\delta_m/\alpha}{\frac{\alpha^2 e^{{\theta L_0}/{\alpha}}}{\alpha^2-\rho_0\rho_1\theta}} (V_0(x))^{\frac{\rho_1\theta}{\alpha}}\nonumber.
\end{align}
Note that $\varphi_\alpha^{(n,\delta_m)}\rightarrow \varphi_\alpha^{(n)}$ as $m\rightarrow\infty$. Thus, letting $m\rightarrow\infty$ in the above equation, we obtain
\interdisplaylinepenalty=0
	\begin{align}
1\leq \varphi^{(n)}_\alpha(\theta,x)
\leq {\frac{\alpha^2 e^{{\theta L_0}/{\alpha}}}{\alpha^2-\rho_0\rho_1\theta}} (V_0(x))^{\frac{\rho_1\theta}{\alpha}}.
\label{eq 4.20}
\end{align}
By using (\ref{eq 4.1}), (\ref{eq 4.2}), (\ref{eq 4.20}), and the PDE satisfied by $\varphi^{(n)}_{\alpha}$ (that is just proven), we say that $\varphi^{(n)}_{\alpha}\in B^1_{V_0,V_1}([0,1]\times S)$ and it is a solution of (\ref{eq 4.11}).
 Thus by closely mimicking the arguments as in Theorem \ref{theo 3.1}, one can easily get the stochastic representation of the solution $\varphi_\alpha^{(n)}$, that is
\interdisplaylinepenalty=0
\begin{align}
\varphi^{(n)}_\alpha(\theta,x)=\inf_{\pi\in \Pi}E^{\pi}_x \biggl[exp\biggl(\theta\int_{0}^{\infty}\int_{A}e^{-\alpha t} c_n(\xi^{(n)}_t,a)\pi(da|\omega,t)dt\biggr)\biggr].\label{eq 4.21}
\end{align}
Step 2: In this step we prove Theorem \ref{theo 4.1}, by passing to the limit as $n \to \infty$. 
Now we will prove that for each $x\in S$, $\{\varphi^{(n)}_\alpha\}_{n\geq 1}$ is equicontinuous on $[0, 1]$.  We consider the following expression, for any given $\pi\in \Pi$, $x\in S$, $\theta,\theta_0\in [0,1]$:
\interdisplaylinepenalty=0
\begin{align*}
\biggl|E^{\pi}_x&\biggl[exp\biggl(\theta\int_{0}^{\infty}\int_{A} e^{-\alpha t} c_n(\xi^{(n)}_t,a)\pi(da|\omega,t)dt\biggr)\biggr]\\
&-E^{\pi}_x \biggl[exp\biggl(\theta_0\int_{0}^{\infty}\int_{A}e^{-\alpha t} c_n(\xi^{(n)}_t,a)\pi(da|\omega,t)dt\biggr)\biggr]\biggr|\\
&\leq K_1,
\end{align*}
where
\interdisplaylinepenalty=0
\begin{align*}
K_1&=E^{\pi}_x \biggl[exp\biggl((\theta\wedge\theta_0)\int_{0}^{\infty}\int_{A} e^{-\alpha t} c_n(\xi^{(n)}_t,a)\pi(da|\omega,t)dt\biggr)\\
&~~~~\times
\biggl(exp\biggl(|\theta_0-\theta|\int_{0}^{\infty}\int_{A} e^{-\alpha t} c_n(\xi^{(n)}_t,a)\pi(da|\omega,t)dt\biggr)-1\biggr)\biggr]\\
&\leq E^{\pi}_x \biggl[exp\biggl((\theta\wedge\theta_0)\int_{0}^{\infty}\int_{A} e^{-\alpha t} c_n(\xi^{(n)}_t,a)\pi(da|\omega,t)dt\biggr)\\
&~~~~\times
\biggl(exp\biggl(\int_{0}^{\infty}\int_{A} e^{-\alpha t} c_n(\xi^{(n)}_t,a)\pi(da|\omega,t)dt \biggr) -1\biggr)|\theta_0-\theta|\biggr]\\
&\leq E^{\pi}_x \biggl[exp\biggl(\int_{0}^{\infty}\int_{A} e^{-\alpha t} c_n(\xi^{(n)}_t,a)\pi(da|\omega,t)dt\biggr)\\
&~~~~\times
\biggl(exp\biggl(\int_{0}^{\infty}\int_{A} e^{-\alpha t} c_n(\xi^{(n)}_t,a)\pi(da|\omega,t)dt\biggr)|\theta_0-\theta|\biggr)\biggr]\\
&= |\theta_0-\theta| \times E^{\pi}_x \biggl[exp\biggl(2\int_{0}^{\infty}\int_{A} e^{-\alpha t} c_n(\xi^{(n)}_t,a)\pi(da|\omega,t)dt\biggr)\biggr]\\
&\leq  |\theta_0-\theta| \times \frac{\alpha e^{{2L_0}/{\alpha}}}{\alpha-\rho_2}M_1^2\biggl(V_1^2(x)+\frac{b_1}{\rho_2}\biggr).
\end{align*}
 Here, the first inequality is according to $e^{bz}-1\leq (e^b-1)z$ for all $z\in [0,1]$ and $b>0$ and the last inequality  follows from (\ref{eq 3.6}). Therefore, we have
 \interdisplaylinepenalty=0
 \begin{align}
|\varphi^{(n)}_\alpha(\theta_0,x)-\varphi^{(n)}_\alpha(\theta,x)|\nonumber
&\leq \sup_{\pi\in \Pi} |\theta_0-\theta| \times \frac{\alpha e^{{2L_0}/{\alpha}}}{\alpha-\rho_2}M_1^2\biggl(V_1^2(x)+\frac{b_1}{\rho_2}\biggr)\nonumber\\
&=|\theta_0-\theta| \times \frac{\alpha e^{{2L_0}/{\alpha}}}{\alpha-\rho_2}M_1^2\biggl(V_1^2(x)+\frac{b_1}{\rho_2}\biggr).\label{eq 4.22}
\end{align}
By measurable selection theorem, [\cite{BS},Proposition 7.33], there exists a measurable function $f^{*}_n:[0,1]\times S\rightarrow A$ such that 
\begin{align}
&\inf_{a\in A(x)}\biggl[\theta c_n(x,a)\varphi_\alpha(\theta,x)+\int_{S}q^{(n)}(dy|x,a)\varphi_\alpha(\theta,y)\biggr]\nonumber\\
&=\biggl[\theta c_n(x,f^{*}_n(\theta,x))\varphi_\alpha(\theta,x)+\int_{S}q^{(n)}(dy|x,f^{*}_n(\theta,x))\varphi_\alpha(\theta,y)\biggr]. \label{eq 4.23} 
\end{align}
Let
\begin{equation*}
\pi^{*}_n:\mathbb{R}_+\times S \to P(A)  \
\end{equation*}
be defined by
	\begin{align*}
&\pi^{*}_n(\cdot|t,x)=I_{\{\hat{f}^{*}_n(t,x)\}}(\cdot),~\text{where}~\hat{f}_n^{*}:[0,1]\times S\rightarrow A,\\&\quad~\text{be a measurable mapping, defined by}~ \hat{f}_n^{*}(t,x):=f_n^{*}(\theta e^{-\alpha t},x).
\end{align*}
Hence by equation (\ref{eq 4.11}), we have a.e. $\theta\in [0,1]$ and $\forall x\in S$, we have
	\begin{align}
\left\{ \begin{array}{lllll}\alpha\theta\frac{\partial\varphi^{(n)}_\alpha}{\partial \theta}(\theta,x)
&=\displaystyle\biggl[\int_{S}q^{(n)}(dy|x,f^{*}_n(\theta,x))\varphi_\alpha^{(n)}(\theta,y)+\theta c_n(x,f^{*}_n(\theta,x))\varphi^{(n)}_\alpha(\theta,x)\biggr]\\
1\leq \varphi^{(n)}_\alpha(\theta,x)&\leq {\frac{\alpha^2 e^{{\theta L_0}/{\alpha}}}{\alpha^2-\rho_0\rho_1\theta}}(V_0(x))^{\frac{\rho_1\theta}{\alpha}}~~~ \forall~(\theta,x)\in [0,1]\times S.\label{eq 4.24}
\end{array}\right.
\end{align}
Since $c_n\geq 0$, by (\ref{eq 4.21}), we say $\varphi^{(n)}_\alpha(\theta,x)$ is increasing in $\theta$. Also we know that $\varphi^{(n)}_\alpha(\theta,x)$ is differentiable a.e. with respect to $\theta\in [0,1]$. So
\begin{align}
\frac{\partial\varphi^{(n)}_\alpha}{\partial \theta}(\theta,x)\geq 0~~\text{for a.e.}~\theta. \label{eq 4.25}
\end{align} 
So, by (\ref{eq 4.1}), (\ref{eq 4.2}) and (\ref{eq 4.24}), for all $x\in S$ and for a.e. $\theta$, we have 
	\begin{align}
\left\{ \begin{array}{lllll}&-\alpha\theta\frac{\partial\varphi^{(n)}_\alpha}{\partial \theta}(\theta,x)+\displaystyle\biggl[\int_{S}q^{(n-1)}(dy|x,f^{*}_n(\theta,x))\varphi_\alpha^{(n)}(\theta,y)+\theta c_{n-1}(x,f^{*}_n(\theta,x))\varphi^{(n)}_\alpha(\theta,x)\biggr]\leq 0\\
&\quad\quad\text{if}~x\in S_{n-1}\label{eq 4.26}
\end{array}\right.
\end{align}
and 
	\begin{align}
\left\{ \begin{array}{llll}&-\alpha\theta\frac{\partial\varphi^{(n)}_\alpha}{\partial \theta}(\theta,x)+\displaystyle\biggl[\int_{S}q^{(n-1)}(dy|x,f^{*}_n(\theta,x))\varphi_\alpha^{(n)}(\theta,y)+\theta c_{n-1}(x,f^{*}_n(\theta,x))\varphi^{(n)}_\alpha(\theta,x)\biggr]\\
&=-\alpha\theta\frac{\partial\varphi^{(n)}_\alpha}{\partial \theta}(\theta,x)\leq 0\\
&\quad\quad\text{if}~x\notin S_{n-1}~\text{(by~(\ref{eq 4.25}))} .\label{eq 4.27 }
\end{array}\right.
\end{align}
So, by Feynman-Kac formula, we get
\begin{align}
E^{\pi^{*}_n}_x\biggl[exp\biggl(\theta\int_{0}^{\infty}\int_{A}e^{-\alpha t}c_{n-1}(\xi^{(n-1)}_t,a)\pi^{*}_n(da|t,\xi^{(n-1)}_t)dt\biggr)\biggr]\leq \varphi^{(n)}_\alpha(\theta,x)~\text{for all}~(\theta,x)\in [0,1]\times S.\label{eq 4.28}
\end{align}
Also using (\ref{eq 4.11}) and Feynman-Kac formula (see (\ref{eq 3.9}) and (\ref{eq 3.14})), we have
\begin{align}
\varphi^{(n-1)}_\alpha(\theta,x)\leq E^{\pi^{*}_n}_x\biggl[exp\biggl(\theta\int_{0}^{\infty}\int_{A}e^{-\alpha t}c_{n-1}(\xi^{(n-1)}_t,a)\pi^{*}_n(da|t,\xi^{(n-1)}_t)dt\biggr)\biggr].\label{eq 4.29}
 \end{align}
By (\ref{eq 4.28}) and (\ref{eq 4.29}), we have $\varphi^{(n-1)}_\alpha(\theta,x)\leq \varphi^{(n)}_\alpha(\theta,x).$\\\\ 
Hence $\varphi^{(n)}_\alpha(\theta,x)$ is increasing in $n$ for any $(\theta,x)\in [0,1]\times S$. Now from (\ref{eq 4.22}), we know that for each $x\in S$, $\varphi^{(n)}(\cdot,x)$ is Lipschitz continous in $\theta\in [0, 1]$. Also, $\varphi^{(n)}_\alpha(\theta,x)$ is increasing as $n\rightarrow \infty$ for any $(\theta,x)\in [0,1]\times S$ and bounded above (by (\ref{eq 4.20})), therefore there exists a function $\varphi_\alpha$ on $[0,1]\times S$ that is continuous with respect to $\theta\in [0,1]$, such that along a subsequence $n_k\rightarrow \infty$, we have $\lim_{n_k\rightarrow\infty}\varphi^{(n_k)}_\alpha(\theta,x)=\varphi_\alpha(\theta,x)$ and this convergence is uniform in $\theta\in [0,1]$ for each fixed $x\in S$.
 Moreover, by (\ref{eq 4.20}), we have
\begin{align*}
1\leq\varphi_\alpha(\theta,x)
&\leq {\frac{\alpha^2 e^{{\theta L_0}/{\alpha}}}{\alpha^2-\rho_0\rho_1\theta}} (V_0(x))^{\frac{\rho_1\theta}{\alpha}}.
\end{align*}
As the proof of equation (\ref{eq 4.11}) in the step 1 (starting from the first equality of (\ref{eq 4.13})), we say that  $\varphi_{\alpha}$ is a solution to the HJB equation (\ref{eq 3.1}). Also by (\ref{eq 3.1}) and Assumption \ref{assm 2.1}, we have $$\biggl|\displaystyle{\inf_{a\in A(x)}\biggl[\int_{S}q(dy|x,a)\varphi_\alpha(\theta,y)+\theta c(x,a)\varphi_\alpha(\theta,x)\biggr]}\biggr|\leq {\frac{\alpha^2 e^{{\theta L_0}/{\alpha}}}{\alpha^2-\rho_0\rho_1\theta}}[\rho_0+2M_0+\rho_1+L_0]M_1V_1(x).$$ So, $|\alpha\theta \frac{\partial\varphi_\alpha}{\partial \theta}(\theta,x)|\leq {\frac{\alpha^2 e^{{\theta L_0}/{\alpha}}}{\alpha^2-\rho_0\rho_1\theta}} [\rho_0+2M_0+\rho_1+L_0]M_1V_1(x)$. Hence $\varphi_{\alpha}\in B^1_{V_0,V_1}([0,1]\times S)$. Finally, the uniqueness of $\varphi_\alpha(\theta,x)$ follows from the stochastic representation in Theorem \ref{theo 3.1}.
\end{proof}
	
\section{The existence of optimal control}
In this section, we present the main result of this article. Here we show the existence of an optimal control.
\begin{thm}\label{SPE}
	Suppose that Assumptions \ref{assm 2.1}, \ref{assm 2.2} and \ref{assm 3.1} are satisfied. Then, the following assertions hold.
	\begin{enumerate}
		\item The HJB equation (\ref{eq 3.1}) has a unique solution $\varphi_\alpha\in B^1_{V_0,V_1}([0,1]\times S)$ and the solution admits the following representation
		\interdisplaylinepenalty=0
\begin{align*}
			1\leq \varphi_\alpha(\theta,x)
			&=\inf_{\pi\in \Pi}E^{\pi}_x\biggl[exp\biggl(\theta\int_{0}^{\infty}\int_{A}e^{-\alpha t}c(\xi_t,a)\pi(da|\omega,t)dt\biggr)\biggr]\\
			&\leq {\frac{\alpha^2 e^{{\theta L_0}/{\alpha}}}{\alpha^2-\rho_0\rho_1\theta}} (V_0(x))^{\frac{\rho_1\theta}{\alpha}}.
			\end{align*}
		\item There exists a  measurable function $f^*: [0,1]\times S \to A$ such that  
\interdisplaylinepenalty=0
		\begin{align}
		\alpha \theta \frac{\partial\varphi_\alpha}{\partial\theta}(\theta,x)
		&=\biggl[\int_{S}q(dy|x,f^{*}(\theta,x))\varphi_\alpha(\theta,y)+\theta c(x,f^{*}(\theta,x))\varphi_\alpha(\theta,x)\biggr]\nonumber\\
		 \text{a.e.}~\theta\in [0,1].\label{eq 5.1}
		\end{align}
		\item Furthermore an optimal Markov control for the cost criterion (\ref{eq 2.5}) exists and is given by $$\tilde{\pi}^*(\cdot|t,x):=I_{\{\hat{f}(t,x)\}}(\cdot),  \hat{f}(t,x):=f^*(\theta e^{-\alpha t},x), $$		
		where $f^*$ satisfies (\ref{eq 5.1}).
		
		\end{enumerate}
		\end{thm}
	\begin{proof}
		Part (1) follows from Theorems \ref{theo 3.1} and \ref{theo 4.1}.	To prove (2), for each given $(\theta,x)\in[0,1]\times S$, by \cite{HL2}, we have the continuity of the function $$ \eta(x,\theta,a):=\int_{S}q(dy|x,a)\varphi_\alpha(\theta,y)+\theta c(x,a)\varphi_\alpha(\theta,x)$$ in $a\in A(x)$.
		 Thus, the measurable selection theorem [\cite{BS}, Proposition 7.33] ensured the existence of a measurable function $f^{*}$ satisfying (\ref{eq 5.1}), and so (2) follows.  Moreover for any $f^{*}$ satisfying (\ref{eq 5.1}), from the proof of Theorem \ref{theo 3.1}, we have $\inf_{\pi\in\Pi} \tilde{J}_\alpha(x,\theta,{\pi})=\tilde{J}_\alpha(x,\theta,\tilde{\pi}^{*})= \varphi_\alpha(\theta,x)$, which together with (\ref{eq 2.5}), (\ref{eq 2.6}) and part (1), we have $\inf_{\pi\in\Pi}\mathscr{J}_\alpha(x,\theta,{\pi})=\mathscr{J}_\alpha(x,\theta,\tilde{\pi}^{*})=\frac{1}{\theta}\ln\tilde{J}_\alpha(x,\theta,\tilde{\pi}^{*})= \frac{1}{\theta}\ln\varphi_\alpha(\theta,x).$
		Hence $\tilde{\pi}^{*}$ is an optimal Markov control.   
  		\end{proof}
  		 \section{Application and example}
  	In this section, we verify the above assumptions with one example, where the transition and cost rates are unbounded. \\
  	
  	\begin{example}
  		\textbf{The Gaussian Model:}
  		Suppose a hunter is hunting outside his house for his
manager. Suppose the house is at state 0. A positive state represents the distance from the house to the right, and a negative state represents the distance from the house to the left. Let $S=\mathbb{R}$. If the current postion is $x\in S$, the hunter takes a action $a\in A(x)$, then after an exponentially distributed travel time with rate $\lambda(x,a)>0$, the hunter reaches the new position, and the travel distance follows the normal distribution with mean $x$ and variance $\sigma$. (Or we can interpret $\lambda(x,a)$ as the total jump intensity that is an arbitrary measurable positive-valued function on $S\times A$, and the distribution of the state after a jump from $x\in S$ is normal with the variance $\sigma$ and expectation $x$.) Also assume that the hunter receives a payoff $c(x,a)$ from his manager for each unit of time he spends there. 
  		Let us consider the model as $A_2:=\{S,(A,A(x),x\in S),c(x,a),q(dy|x,a)\}$, where $S=(-\infty,\infty)$. For each $D\in \mathscr{B}(S)$, the transition rate is
  		\begin{align}
  		q(D|x,a)=\lambda(x,a)\bigg[\int_{y\in D}\frac{1}{\sqrt{2\pi}\sigma}e^{-\frac{(y-x)^2}{2\sigma^2}}dy-\delta_x(D)\bigg],~x\in S,a\in A(x).
  		\end{align}
  		To ensure the existence of an optimal Markov control for the model, we consider the following hypotheses.
  		\begin{enumerate}
  			\item [(I)] For each fixed $x\in S$, $\lambda(x,a)$ is continuous in $a\in A(x)$ and there exists a positive constant $M$ such that $0<\sup_{a\in A(x)}\lambda(x,a)\leq M(x^2+1)$ and $M<\frac{\alpha}{3780(\sigma^8+\sigma^6+\sigma^4+\sigma^2)}$.
  			\item [(II)] For each $x\in S$, the cost rate $c(x,a)$ is nonnegative and continuous in $a\in A(x)$ and there exists constant $0<\rho_1<\min\{\alpha,\frac{\alpha^2}{M\sigma^2}\}$ such that
  			$$\sup_{a\in A(x)}c(x,a)\leq \rho_1 \log(1+x^2).$$
  			\item [(III)] For each fixed $x\in S$, $A(x)$ is a compact subset of the Borel spaces $A$.
  		\end{enumerate} 
  	\end{example}
  	\begin{proposition}	\label{Prop 5.1}
  		Under conditions (I)-(III), the above controlled system satisfies the Assumptions \ref{assm 2.1}, \ref{assm 2.2}, and \ref{assm 3.1}. Hence by Theorem \ref{SPE}, there exists an optimal Markov control for this model.
  	\end{proposition}
  	\begin{proof}
  		We know $\frac{1}{\sqrt{2\pi}\sigma}\int_{-\infty}^{\infty}(y-x)^{2k+1}e^{-\frac{(y-x)^2}{2\sigma^2}}dy=0$ and $\frac{1}{\sqrt{2\pi}\sigma}\int_{-\infty}^{\infty}(y-x)^{2k}e^{-\frac{(y-x)^2}{2\sigma^2}}dy=1\cdot 3\cdots(2k-1)\sigma^{2k}$ for all $k=0,1\cdots.$\\
  		We first verify Assumption \ref{assm 2.1}.
  		Let $V_0(x)=x^2+1$.
  		\begin{align}\label{eq 5.3}
  		\int_{S}V_0(y)q(dy|x,a)&=\lambda(x,a)\bigg[\frac{1}{\sqrt{2\pi}\sigma}\int_{-\infty}^{\infty}(y^2+1)e^{-\frac{(y-x)^2}{2\sigma^2}}dy-(x^2+1)\bigg]\nonumber\\
  		&=\lambda(x,a)\sigma^2\nonumber\\&\leq M \sigma^2V_0(x).
  		\end{align}
  		Let $\rho_0=M \sigma^2$. Then $	\int_{S}V_0(y)q(dy|x,a)\leq \rho_0V_0(x).$ Now
  		\begin{align*}
  		q^{*}(x)=\sup_{a\in A(x)}q_x(a)=\sup_{a\in A(x)}\lambda(x,a)\leq M(x^2+1)=M V_0(x)~\forall~x\in S. 
  		\end{align*}
  		 Now by condition (II), we can write
  		\begin{align*}
  		\sup_{a\in A(x)}c(x,a)\leq \rho_1 \log(1+x^2)+M.
  		\end{align*}
  		Also by condition (II), $0<\rho_1<\min\{\alpha,\rho_0^{-1}\alpha^2\}$.
  		Hence Assumption \ref{assm 2.1} is verified.\\
  		Now we verify Assumption \ref{assm 2.2}.\\
  		Let $V_1(x)=x^4+1$. Then
  		for any $x\in S$, $a\in A(x)$,
  		\begin{align*}
  		\int_{S}q(dy|x,a)V^2_1(y)
  		&=\lambda(x,a)\biggl[\frac{1}{\sqrt{2\pi}\sigma}\int_{-\infty}^{\infty}(y^4+1)^2e^{-\frac{(y-x)^2}{2\sigma^2}}dy-(x^4+1)^2\biggr]\\
  		&=\lambda(x,a)(105\sigma^8+420x^2\sigma^6+210 x^4\sigma^4+6\sigma^4+12\sigma^2x^2+28x^6\sigma^2)\\
  		&\leq 420\lambda(x,a)(\sigma^8+\sigma^6+\sigma^4+\sigma^2)(x^6+x^4+x^2+1)\\
  		&\leq 420\lambda(x,a)(\sigma^8+\sigma^6+\sigma^4+\sigma^2)(3x^6+3)\\
  		&\leq 1260M(\sigma^8+\sigma^6+\sigma^4+\sigma^2)(x^6+1)(1+x^2)\\
  		&\leq  3780M(\sigma^8+\sigma^6+\sigma^4+\sigma^2)(x^4+1)^2\\
  		&\leq  3780M(\sigma^8+\sigma^6+\sigma^4+\sigma^2)V_1^2(x)+1.
  		\end{align*}
  		Now take $\rho_2=3780M(\sigma^8+\sigma^6+\sigma^4+\sigma^2)$. Then by condition (I), we have $0<\rho_2<\alpha$.
  		Also, $(1+x^2)^2\leq 2(1+x^4)$ for all $x\in S$. Let $M_1=2$, then $V_0^2(x)\leq M_1V_1(x)$. 
  		Hence, Assumption \ref{assm 2.2} is verified.
  		Now by conditions (I) and (II), $q(\cdot|x,a)$ and $c(x,a)$ are continuous in $a\in A(x)$. Now by (\ref{eq 5.3}), $	\int_{S}V_0(y)q(dy|x,a)$ is continuous in $a\in A(x)$. Hence Assumption \ref{assm 3.1} is also verified. So, by Theorem \ref{SPE}, we say that there exists an optimal Markov control for this model.
  	\end{proof}
   \bibliographystyle{elsarticle-num}

\end{document}